\documentclass[12pt,english]{article}
\usepackage[T1]{fontenc}
\usepackage[latin9]{inputenc}
\usepackage[a4paper]{geometry}
\usepackage{mathrsfs}
\usepackage{amscd}
\usepackage{amsmath}
\usepackage{amsthm}
\usepackage{amssymb}

\makeatletter
\numberwithin{equation}{section}
\theoremstyle{remark}
\newtheorem*{acknowledgement*}{\protect\acknowledgementname}
\theoremstyle{plain}
\newtheorem{thm}{\protect\theoremname}[section]
\theoremstyle{plain}
\newtheorem{lem}[thm]{\protect\lemmaname}
\theoremstyle{plain}
\newtheorem{prop}[thm]{\protect\propositionname}
\theoremstyle{definition}
\newtheorem{defn}[thm]{\protect\definitionname}
\theoremstyle{plain}
\newtheorem{cor}[thm]{\protect\corollaryname}

\date{}

\makeatother

\usepackage{babel}
\providecommand{\acknowledgementname}{Acknowledgement}
\providecommand{\corollaryname}{Corollary}
\providecommand{\definitionname}{Definition}
\providecommand{\lemmaname}{Lemma}
\providecommand{\propositionname}{Proposition}
\providecommand{\theoremname}{Theorem}

\begin{document}
\title{Gelfand hypergeometric functions as solutions to the 2-dimensional
Toda-Hirota equations II }
\author{Hironobu Kimura,\\
 Department of Mathematics, Graduate School of Science and\\
 Technology, Kumamoto University}

\maketitle
\global\long\def\R{\mathbb{R}}%
 
\global\long\def\pa{\partial}%
\global\long\def\al{\alpha}%
\global\long\def\be{\beta}%
 
\global\long\def\ga{\gamma}%
 
\global\long\def\de{\delta}%
 
\global\long\def\expo{\mathrm{exp}}%
 
\global\long\def\f{\varphi}%
 
\global\long\def\W{\Omega}%
 
\global\long\def\w{\omega}%
 
\global\long\def\lm{\lambda}%
 
\global\long\def\te{\theta}%
 
\global\long\def\C{\mathbb{C}}%
 
\global\long\def\Z{\mathbb{Z}}%
 
\global\long\def\Ps{\mathbb{P}}%
 
\global\long\def\De{\Delta}%
 
\global\long\def\cbatu{\mathbb{C}^{\times}}%
 
\global\long\def\La{\Lambda}%
 
\global\long\def\vt{\vartheta}%
 
\global\long\def\G{\Gamma}%
 
\global\long\def\GL#1{\mathrm{GL}(#1)}%
 
\global\long\def\gras{\mathrm{GM}}%
 
\global\long\def\ep{\epsilon}%
  
\global\long\def\diag{\mathrm{diag}}%
 
\global\long\def\tr{\,\mathrm{^{t}}}%
 
\global\long\def\re{\mathrm{Re}}%
 
\global\long\def\im{\mathrm{Im}}%
 
\global\long\def\s{\sigma}%
 
\global\long\def\sp{\mathrm{sp}}%
 
\global\long\def\rank{\mathrm{rank}\,}%
 
\global\long\def\tH{\tilde{H}}%
 
\global\long\def\lto{\longrightarrow}%
 
\global\long\def\S{\mathfrak{S}}%
 
\global\long\def\cO{\mathcal{O}}%
 
\global\long\def\gl{\mathfrak{gl}}%
 
\global\long\def\gee{\mathfrak{g}}%
 
\global\long\def\Tr{\,\mathrm{Tr}}%
 
\global\long\def\cS{\mathcal{S}}%
 
\global\long\def\ad{\mathrm{ad}}%
 
\global\long\def\Ad{\mathrm{Ad}}%
 
\global\long\def\ha{\mathfrak{h}}%
 
\global\long\def\sgn{\mathrm{sgn}}%
 
\global\long\def\hlam{H_{\lambda}}%
 
\global\long\def\cP{\mathcal{P}}%
 
\global\long\def\auto{\mathrm{Aut}}%
 
\global\long\def\la{\langle}%
 
\global\long\def\ra{\rangle}%
 
\global\long\def\Ai{\mathrm{Ai}}%
 
\global\long\def\adj{\mathrm{Ad}}%
 
\global\long\def\yn{\mathbf{Y}_{n}}%
 
\global\long\def\ideal{\mathcal{I}}%
 
\global\long\def\hyp#1#2{\, _{#1}F_{#2}}%
\global\long\def\sl{sl_{2}(\C)}%
 
\global\long\def\Jor#1{J(#1)}%
 
\global\long\def\mat#1{\mathrm{Mat}(#1)}%
 
\global\long\def\jro{J^{\circ}}%
 
\global\long\def\jroo{\mathfrak{j}^{\circ}}%
 
\global\long\def\fa{\mathfrak{a}}%
 
\global\long\def\pa{\partial}%
 
\global\long\def\bx{\mathbf{x}}%
 
\global\long\def\Joor#1{J^{\circ}(#1)}%
 
\global\long\def\fj{\mathfrak{j}}%
 
\global\long\def\matt#1{\mathrm{Mat}'(#1)}%
 
\global\long\def\rk{\mathrm{rank}}%
 
\global\long\def\cL{\mathcal{L}}%
 
\global\long\def\by{\mathbf{y}}%
 
\global\long\def\cE{\mathcal{E}}%
 
\global\long\def\cM{\mathcal{M}}%

\begin{abstract}
We construct solutions of the $2$-dimensional Toda-Hirota equation
(2dTHE) expressed by the Gelfand hypergeometric function (Gelfand
HGF) on the Grassmannian $\mathrm{GM}(2,N)$ of confluent or non-confluent
type, which is labeled by a partition of $N$. A system of hyperbolic
equations in $N$ complex variables is obtained from the differential
equations which form a main body of the Gelfand hypergeometric system
and characterize the image of Radon transform. We use the Laplace
sequence of the system of hyperbolic operators to find an elementary
seed solution of the 2dTHE and then use the B\"acklund transformation
to obtain higher solutions expressed in terms of Gelfand HGF. In constructing
the Laplace sequence, the contiguity relations (operators) for the
Gelfand HGF play an important role. 
\end{abstract}
Key Words and Phrases: Toda-Hirota equation, Gelfand hypergeometric
function, Laplace sequence, contiguity relation.

Mathematics Subject Classification; 33C70, 35Q05, 35Q51.

\section{Introduction}

This paper is a continuation of \cite{hiro-kimu}. In the previous
paper, we made clear how the Gelfand hypergeometric function (Gelfand
HGF) on the Grassmannian manifold $\gras(2,N)$ of non-confluent type
provides a solution of the 2-dimensional Toda-Hirota equation (2dTHE):
\begin{equation}
\partial_{x}\partial_{y}\log\tau_{n}=\frac{\tau_{n+1}\tau_{n-1}}{\tau_{n}^{2}},\quad n\in\mathbb{Z},\label{eq:intro-1}
\end{equation}
which is a bilinear form of the 2-dimensional Toda equation. In this
paper, we treat the Gelfand HGF on $\gras(2,N)$ of general type,
namely, together with those of confluent type, and show that they
also give solutions of the 2dTHE. The 2dTHE is an important nonlinear
integrable system and its structure of the solutions was studied from
various viewpoints by many authors \cite{Darboux,Hirota,kame-1,Nakamura-A,Okamoto-2,Olshanetsky,Satuma}.
There are various type of solutions: soliton solutions, Wronskian
determinant type solutions and solutions expressed by special functions,
for example. Here we focus on the solutions of the last type. As for
this type of solutions, there are works by Kametaka \cite{kame-1,kame-2}
and Okamoto \cite{Okamoto-1,Okamoto-2}. They showed that the classical
family of HGF: Gauss HGF, Kummer's confluent HGF, Bessel function,
Hermite-Weber function, give solutions of the 2dTHE. They also showed
that some of the HGFs of 2-variables in the Horn's list \cite{Erdelyi}
give solutions of the 2dTHE. In particular, Appel's HGF $F_{1}$ and
its confluent family $\Phi_{1},\Phi_{2},\Phi_{3}$ are contained in
their result. For example, $F_{1}$ and $\Phi_{1}$ are given by 
\begin{align*}
F_{1}(\al,\be,\be',\ga;x,y) & =\sum_{m,n=0}^{\infty}\frac{(\al)_{m+n}(\be)_{m}(\be')_{n}}{(\ga)_{m+n}m!n!}x^{m}y^{n}\\
 & =C\int_{0}^{1}s^{\alpha-1}(1-s)^{\gamma-\alpha-1}(1-xs)^{-\beta}(1-ys)^{-\beta'}ds,\\
\Phi_{1}(\al,\be,\ga;x,y) & =\sum_{m,n=0}^{\infty}\frac{(\al)_{m+n}(\be)_{m}}{(\ga)_{m+n}m!n!}x^{m}y^{n}\\
 & =C\int_{0}^{1}s^{\alpha-1}(1-s)^{\gamma-\alpha-1}(1-xs)^{-\beta}e^{ys}ds,
\end{align*}
where $(\al)_{m}=\G(\al+m)/\G(\al)$ and $C=\G(\ga)/\G(\al)\G(\ga-\al)$.
Their method is based on the relation of (\ref{eq:intro-1}) to the
Laplace sequence $\{M_{n}\}_{n\in\Z}$ for a simple hyperbolic equation
\[
M=\partial_{x}\partial_{y}+a(x,y)\partial_{x}+b(x,y)\partial_{y}+c(x,y),
\]
which was discussed by Darboux \cite{Darboux} in differential geometry
of surfaces in $\R^{n}$. Write $M$ in the form $M=(\partial_{x}+b)(\partial_{y}+a)-h$,
where $h=a_{x}+ab-c$. If $h=0$, then $M$ becomes a product of the
first order operator. In this sense, $h$ measures the decomposability
of $M$ and is called the invariant of $M$. If $h\neq0$, one can
construct a new operator $M_{1}$ of the same type by making a change
of unknown $u\mapsto u_{1}=(\partial_{y}+a)u$ for $Mu=0$. Apply
the same process to $M_{1}$ if its invariant does not vanish and
get $M_{2}$, and so on. Then, starting from $M_{0}=M$, one may obtain
a sequence of hyperbolic operators $\{M_{n}\}_{n\geq0}$ with the
invariants $h_{n}$ for $M_{n}$, which is called the Laplace sequence.
As will be explained at the last part of Section \ref{subsec:Laplace-seq},
$\{h_{n}\}_{n\geq0}$ satisfies the equation 
\[
\partial_{x}\partial_{y}\log h_{n}=-h_{n+1}+2h_{n}-h_{n-1},
\]
and if $t_{n}$ is determined suitably by $\partial_{x}\partial_{y}\log t_{n}=-h_{n-1}$,
then $\{t_{n}\}_{n\geq0}$ gives a particular solution of the 2dTHE,
which we call a seed solution. Then together with the appropriately
chosen solution $u_{n}$ of $M_{n}u=0$, one obtains a new solution
$\{\tau_{n}\}$ of the 2dTHE in the form $\tau_{n}=t_{n}u_{n}$. The
process $t_{n}\to\tau_{n}$, which gives a new solution $\tau_{n}$
from the old one $t_{n}$, is called the Bäcklund transformation,
see Section \ref{subsec:2DTE-laplace}. So the important point is
to find a good operator $M$. In the previous paper \cite{hiro-kimu},
the key hyperbolic operator was the Euler-Poisson-Darboux operator
(EPD operator) 
\[
M=\partial_{x}\partial_{y}+\frac{\beta}{x-y}\partial_{x}+\frac{\al}{y-x}\partial_{y}.
\]
It played an important role to recognize the Gelfand HGF of non-confluent
type as a solution of the 2dTHE. The function $u_{n}$ in the B\"acklund
transformation was given in terms of the Gelfand HGF. In this paper,
analogous type of hyperbolic operators appear naturally and play a
similar role in constructing solutions of the 2dTHE by the Gelfand
HGF of confluent type. 

The Gelfand HGF on the complex Grassmannian manifold $\mathrm{GM}(2,N)$
is a natural generalization of the HGFs appeared above, and it enables
us a unified approach to understand various aspects of classical HGFs
\cite{Gelfand,Kimura-Haraoka,Kimura-H-T}. It is defined as a Radon
transform of a character $\chi_{\lm}(\cdot\,;\al):\tilde{H}_{\lm}\to\cbatu$
of the universal covering group of the maximal abelian subgroup $H_{\lambda}\subset\mathrm{GL}(N)$.
The group $H_{\lm}$ is of dimension $N$ and specified by a partition
$\lambda=(n_{1},\dots,n_{\ell})$ of $N$. The character $\chi_{\lm}$
depends on $\al=(\al^{(1)},\dots,\al^{(\ell)})\in\C^{N}$ with $\al^{(j)}=(\al_{0}^{(j)},\dots,\al_{n_{j}-1}^{(j)})\in\C^{n_{j}},$
see Section \ref{subsec:Definition-of-HGF} for the detail. When $\lambda=(1,\dots,1)$,
$H_{\lambda}$ is a Cartan subgroup of $\GL N$ and the Gelfand HGF
is of non-confluent type and is essentially Appell's HGF $F_{D}$
of $N-3$ variables. The Gauss HGF and its confluent family: Kummer,
Bessel, Hermite-Weber and Airy, are understood as the Gelfand HGF
on $\mathrm{GM}(2,4)$ corresponding to the partitions $(1,1,1,1),(2,1,1),(2,2),(3,1)$
and $(4)$, respectively \cite{Kimura-Haraoka}, see also Section
\ref{subsec:Classical-HGFs}. 

The Gelfand HGF $F(z;\al)$ is a function on the Zariski open subset
$Z_{\lm}$ of $\mathrm{Mat}(2,N)$ and is defined by the 1-dimensional
complex integral 
\[
F(z;\alpha)=\int_{C(z)}\chi_{\lm}(\vec{s}z;\al)ds,\quad\vec{s}=(1,s),
\]
where $C(z)$ is an appropriate path of integration in $\C$. Here
the $N$-dimesional row vector $\vec{s}z$ is regarded as an element
of $\tilde{H}_{\lm}$. To relate $F(z;\al)$ to the 2dTHE, we consider
the slice $X\subset Z_{\lm}$ : 
\[
X=\left\{ \mathbf{x}=(\bx^{(1)},\dots,\bx^{(\ell)})\in Z_{\lm}\mid\bx^{(j)}=\left(\begin{array}{cccc}
x_{0}^{(j)} & x_{1}^{(j)} & \cdots & x_{n_{j}-1}^{(j)}\\
1 & 0 & \cdots & 0
\end{array}\right)\;(1\leq j\le\ell)\right\} 
\]
and the restriction $F(\bx;\al)$ of $F(z;\al)$ to $X$. Then we
see that $F(\bx;\al)$ satisfies an analogue of EPD equation
\begin{equation}
M^{(i,j)}(\al)u:=\left(D_{n_{i}-1}^{(i)}D_{n_{j}-1}^{(j)}+\frac{\al_{n_{j}-1}^{(j)}}{x_{0}^{(i)}-x_{0}^{(j)}}D_{n_{i}-1}^{(i)}+\frac{\al_{n_{i}-1}^{(i)}}{x_{0}^{(j)}-x_{0}^{(i)}}D_{n_{j}-1}^{(j)}\right)u=0\label{eq:intro-3}
\end{equation}
for any $1\leq i\neq j\leq\ell$, where $D_{k}^{(j)}=\pa/\pa x_{k}^{(j)}$.
The main result, Theorem \ref{thm:main}, asserts in particular the
following. For any $1\leq i\neq j\le\ell$, 

\[
\tau_{m}(x)=C_{m}(\al)t_{m}(\bx;\al)g_{m}(x)F(\bx;\alpha+m(\ep^{(i)}-\ep^{(j)}))
\]
gives a solution of the 2dTHE 
\begin{equation}
D_{n_{i}-1}^{(i)}D_{n_{j}-1}^{(j)}\log\tau_{m}=\frac{\tau_{m+1}\tau_{m-1}}{\tau_{m}^{2}},\quad m\in\mathbb{Z},\label{eq:intro-4}
\end{equation}
where $t_{m}(\bx;\al)$ is the seed solution of (\ref{eq:intro-4})
derived from the Laplace sequence $\{M_{m}^{(i,j)}(\al)\}$ obtained
from $M^{(i,j)}(\al)$, $g_{m}(x)$ is the factor used in taking $M_{m}^{(i,j)}(\al)$
to the normal form, 
\[
C_{m}(\al)=\begin{cases}
\G(\al_{0}^{(j)}+1)/\G(\al_{0}^{(j)}-m+1), & n_{i}=n_{j}=1,\\
(\al_{n_{j}-1}^{(j)})^{m}, & n_{j}\geq2,
\end{cases}
\]
and $\alpha+m(\ep^{(i)}-\ep^{(j)})$ is obtained from $\al$ by the
change $\text{\ensuremath{\al_{0}^{(i)}}}\mapsto\al_{0}^{(i)}+m,\text{\ensuremath{\al_{0}^{(j)}}}\mapsto\al_{0}^{(j)}-m$.
In showing this result, the contiguity relations for $F(z;\alpha)$
plays an essential role. As for the contiguity of the Gelfand HGF,
see \cite{Kimura-H-T}. 

This paper is organized as follows. In Section 2, we recall the facts
on the Laplace sequence of hyperbolic operators and its relation to
the 2-dimensional Toda equation satisfied by the invariants. The link
to 2dTHE is also discussed. In Section 3, we recall the definition
of the Gelfand HGF of type $\lm$ on $\mathrm{GM}(2,N)$ and its covariance
with respect to the group action $\mathrm{GL}(2)\curvearrowright Z_{\lm}\curvearrowleft H_{\lm}$,
and give the system of differential equations (Gelfand HGS) satisfied
by it. The HGS consists of the equations which come from the above
covariance property and the system of the 2nd order equations which
characterizes the image of Radon transform and forms a main body of
the Gelfand HGS. The contiguity relations of the Gelfand HGF are explained
in Section \ref{subsec:Contiguity}. Sections 4,5, we consider only
the main body of the Gelfand HGS and study the effect of the contiguity
operators and the action of $SL(2,\C)$ on the space of solutions
of the system.

In Section \ref{sec:Restr-slice}, we give the explicit form of the
operators on the slice $X$ derived form the contiguity operators
and the hyperbolic equations (\ref{eq:intro-3}) related to these
contiguity operators. 

In Section \ref{sec:Lap-gelfand}, we consider the Laplace sequence
for the hyperbolic operator obtained in the previous section and determine
the explicit form of the seed solution of the 2dTHE (\ref{eq:intro-4}).
In Section \ref{sec:HGF-2dTHE}, we combine the results obtained in
Section \ref{sec:Lap-gelfand} with that in Section 2 to give Theorem
\ref{thm:main}, the main theorem of this paper. In the last section,
we the data for the Gelfand HGF on $\gras(2,4)$ to relate it to the
c to relate it to the classical HGF family. We also give the 2nd order
differential equations on $X$ derived from the main body of the Gelfand
HGS corresponding to the classical HGFs. 
\begin{acknowledgement*}
I thank Professor Kazuo Okamoto for valuable comments. This work was
supported by JSPS KAKENHI Grant Number JP19K03521. 
\end{acknowledgement*}

\section{\label{sec:LaplaceSeq}Laplace sequence and Toda lattice }

\subsection{\label{subsec:Laplace-seq}Laplace sequence}

We explain briefly the Laplace sequence and the 2dTHE. As for the
detailed account of them, see the previous paper \cite{hiro-kimu}.
Let $x,y$ be the complex coordinates of $\C^{2}$, $\W\subset\C^{2}$
be a simply connected domain, and $\cO(\W)$ be the set of holomorphic
functions on $\W$. Consider the hyperbolic differential equation:
\begin{equation}
Mu=\left(\pa_{x}\pa_{y}+a(x,y)\pa_{x}+b(x,y)\pa_{y}+c(x,y)\right)u=0,\label{eq:laplace-1}
\end{equation}
where $\pa_{x}=\pa/\pa x,\pa_{y}=\pa/\pa y$ and $a,b,c\in\text{\ensuremath{\cO}(\ensuremath{\W})}$.
The operator $M$ is written in the form
\begin{equation}
M=(\pa_{x}+b)(\pa_{y}+a)-h,\label{eq:laplace-1-1}
\end{equation}
 or 
\begin{equation}
M=(\pa_{y}+a)(\pa_{x}+b)-k\label{eq:laplace-1-2}
\end{equation}
with $h=a_{x}+ab-c,\quad k=b_{y}+ab-c$, where $a_{x}=\pa_{x}a,b_{y}=\pa_{y}b$.
The functions $h=h(x,y),k=k(x,y)$ are called the \emph{invariants}
of the operator $M$. 
\begin{lem}
\label{lem:normal-form} \cite{hiro-kimu} For the operator $M$ above
and for a function $f\in\cO(\W)$ such that $1/f\in\cO(\W)$, define
the operator $M'$ by
\[
M'=f^{-1}\cdot M\cdot f=\pa_{x}\pa_{y}+a'\pa_{x}+b'\pa_{y}+c'.
\]
Then $M'$ is given in terms of $F=\log f$ by
\begin{align}
a' & =a+F_{y},\nonumber \\
b' & =b+F_{x},\label{eq:laplace-1-4}\\
c' & =c+aF_{x}+bF_{y}+F_{x}F_{y}+F_{x,y}.\nonumber 
\end{align}
The invariants of $M$ coincide with those of $M'$. 
\end{lem}

\begin{lem}
\label{lem:norm-1}Take $f$ satisfying $b+\pa_{x}\log f=0$, namely,
$f=\exp F,F=-\int^{x}b(t,y)dt$. Then $M'=f^{-1}\cdot M\cdot f$ has
the form $M'=\pa_{x}\pa_{y}+a'\pa_{x}+c'$ with 
\[
a'=a+F_{y},\quad c'=c+aF_{x}+bF_{y}+F_{x,y}+F_{x}F_{y}.
\]
In this case $a'$ and $c'$ are related to the invariants $h,k$
as 
\begin{equation}
a'_{x}=h-k,\quad c'=-k.\label{eq:laplace-8}
\end{equation}
\end{lem}

\begin{lem}
\label{lem:Lap_+}For the operator $M$ given by (\ref{eq:laplace-1}),
assume that $h(x,y)\neq0$ for any $(x,y)\in\W$. Then, by the change
of unknown $u\mapsto u_{+}$: 
\begin{equation}
u_{+}=\mathscr{L}_{+}u:=(\pa_{y}+a)u,\label{eq:laplace-1-3}
\end{equation}
the equation (\ref{eq:laplace-1}) is transformed to 
\begin{equation}
M_{+}u_{+}=\left(\pa_{x}\pa_{y}+a_{+}\pa_{x}+b_{+}\pa_{y}+c_{+}\right)u_{+}=0,\label{eq:laplace-2}
\end{equation}
 where
\begin{align}
a_{+} & =a-\pa_{y}\log h,\nonumber \\
b_{+} & =b,\label{eq:laplace-3}\\
c_{+} & =c-a_{x}+b_{y}-b\,\pa_{y}\log h.\nonumber 
\end{align}
The invariants $h_{+},k_{+}$ of $M_{+}$ are related to those of
$M$ as 

\begin{equation}
h_{+}=2h-k-\pa_{x}\pa_{y}\log h,\quad k_{+}=h.\label{eq:laplace-3-0}
\end{equation}
\end{lem}

Using the expression (\ref{eq:laplace-1-2}) for $M$, we can obtain
the following result in a similar way as in Lemma \ref{lem:Lap_+}.

\begin{lem}
\label{lem:Lap_-}For the operator $M$ given by (\ref{eq:laplace-1}),
assume that $k(x,y)\neq0$ for any $(x,y)\in\W$. Then, by the change
of unknown $u\mapsto u_{-}$:
\begin{equation}
u_{-}=\mathscr{L}_{-}u:=(\pa_{x}+b)u,\label{eq:laplace-3-1-1}
\end{equation}
 the equation (\ref{eq:laplace-1}) is transformed to 
\[
M_{-}u_{-}=\left(\pa_{x}\pa_{y}+a_{-}\pa_{x}+b_{-}\pa_{y}+c_{-}\right)u_{-}=0,
\]
where 
\begin{align}
a_{-} & =a,\nonumber \\
b_{-} & =b-\pa_{x}\log k,\label{eq:laplace-3-2}\\
c_{-} & =c+a_{x}-b_{y}-a\,\pa_{x}\log k.\nonumber 
\end{align}
 The invariants $h_{-},k_{-}$ of $M_{-}$ are related to those of
$M$ as

\begin{equation}
h_{-}=k,\quad k_{-}=2k-h-\pa_{x}\pa_{y}\log k.\label{eq:laplace3-3}
\end{equation}
\end{lem}

The expressions (\ref{eq:laplace-1-1}) and (\ref{eq:laplace-1-2})
for $M$ imply that 
\begin{equation}
\left(\mathscr{L}_{-}\circ\mathscr{L}_{+}\right)u=h\cdot u,\quad\left(\mathscr{L}_{+}\circ\mathscr{L}_{-}\right)u=k\cdot u\label{eq:laplace-3-4}
\end{equation}
holds for any solution $u$ of $Mu=0$.

As a consequence of Lemmas \ref{lem:Lap_+}, \ref{lem:Lap_-}, we
have the sequence $\{M_{n}\}_{n\in\Z}$ of hyperbolic differential
operators starting from $M_{0}:=M$: 
\begin{equation}
\cdots\leftarrow M_{-n}\leftarrow\cdots\leftarrow M_{-1}\leftarrow M_{0}\to M_{1}\to\cdots\to M_{n}\to\cdots,\label{eq:laplace-3-3}
\end{equation}
where, for $n\geq1$, $M_{n}$ is obtained from $M_{n-1}$ by applying
Lemma \ref{lem:Lap_+} under the condition that the invariant $h$
of $M_{n-1}$ satisfies $h\neq0$, and $M_{-n}$ is obtained from
$M_{-n+1}$ by applying Lemma \ref{lem:Lap_-} under the condition
that the invariant $k$ of $M_{-n+1}$ satisfies $k\neq0$. The sequences
$\{M_{n}\}_{n\geq0}$ or $\{M_{n}\}_{n\leq0}$ or $M_{n}\}_{n\in\Z}$
is called the \emph{Laplace sequence} obtained from $M_{0}$. The
invariants of $M_{n}$ will be denoted as $h_{n},k_{n}$. From now
on, in considering Laplace sequences, we tacitly assume that the invariants
do not vanish.

The following results are the consequences of Lemmas \ref{lem:Lap_+},
\ref{lem:Lap_-}.
\begin{prop}
\label{prop:lap-1}For the Laplace sequence $\{M_{n}\}_{n\in\Z_{\geq0}}$,
$M_{n}=\pa_{x}\pa_{y}+a_{n}\pa_{x}+b_{n}\pa_{y}+c_{n}$, the operator
$M_{n+1}$ and its invariants are determined from $M_{n}$ as 

\begin{align}
a_{n+1} & =a_{n}-\pa_{y}\log h_{n},\nonumber \\
b_{n+1} & =b_{n},\label{eq:laplace-4}\\
c_{n+1} & =c_{n}-\pa_{x}a_{n}+\pa_{y}b_{n}-b_{n}\pa_{y}\log h_{n},\nonumber 
\end{align}
and 

\begin{equation}
h_{n+1}=2h_{n}-k_{n}-\pa_{x}\pa_{y}\log h_{n},\;k_{n+1}=h_{n}.\label{eq:laplace-4-1}
\end{equation}
\end{prop}

\begin{prop}
\label{prop:lap-2}For the Laplace sequence $\{M_{n}\}_{n\in\Z_{\leq0}}$,
$M_{n}=\pa_{x}\pa_{y}+a_{n}\pa_{x}+b_{n}\pa_{y}+c_{n}$, the operator
$M_{n-1}$ and its invariants are determined from $M_{n}$ as 

\begin{align}
a_{n-1} & =a_{n},\nonumber \\
b_{n-1} & =b_{n}-\pa_{x}\log k_{n},\label{eq:laplace-5}\\
c_{n-1} & =c_{n}+\pa_{x}a_{n}-\pa_{y}b_{n}-a_{n}\pa_{x}\log k_{n},\nonumber 
\end{align}
 and 

\begin{equation}
h_{n-1}=k_{n},\;k_{n-1}=2k_{n}-h_{n}-\pa_{x}\pa_{y}\log k_{n}.\label{eq:laplace-5-1}
\end{equation}
 
\end{prop}

Put $r_{n}=-k_{n}=-h_{n-1}$. Then the relations (\ref{eq:laplace-4-1})
and (\ref{eq:laplace-5-1}) are expressed as 
\begin{equation}
\pa_{x}\pa_{y}\log r_{n}=r_{n+1}-2r_{n}+r_{n-1},\quad n\in\Z.\label{eq:toda-2}
\end{equation}
 This recurrence relation is called the \emph{2-dimensional Toda equation}
(2dTE). In Section \ref{subsec:2DTE-laplace}, we consider another
form of 2dTE.

\subsection{\label{subsec:The-sequence-normal}Laplace sequence of hyperbolic
operators in the normal form}

To relate the Laplace sequence to another form of the 2dTE, we discuss
the reduction of the operator 
\begin{equation}
M=\pa_{x}\pa_{y}+a\pa_{x}+b\pa_{y}+c,\label{eq:laplace-6-1}
\end{equation}
to the normal form $M'=\pa_{x}\pa_{y}+a'\pa_{x}+c'$, which is obtained
from (\ref{eq:laplace-6-1}) by eliminating the term $b\pa_{y}$ by
considering $M'=f^{-1}\cdot M\cdot f$ with an appropriate function
$f$. This corresponds to consider the change of unknown $u\to v=f^{-1}u$
for the system $Mu=0$. To find such $f$, note the expression (\ref{eq:laplace-1-4})
for $b'$ in Lemma \ref{lem:normal-form}.

Suppose we are given $M_{0}=\pa_{x}\pa_{y}+a_{0}\pa_{x}+c_{0}$ in
the normal form. We construct the sequence $\{M_{n}\}_{n\in\Z}$ consisting
of the operators 
\[
M_{n}=\pa_{x}\pa_{y}+a_{n}\pa_{x}+c_{n},
\]
such that $M_{n+1}$ is obtained from $M_{n}$ by the process given
in Lemma \ref{lem:Lap_+}. Let $\{M_{n}\}_{n\geq0}$ be the Laplace
sequence constructed from $M_{0}$ by Proposition \ref{prop:lap-1}.
Then (\ref{eq:laplace-4}) implies that the operators $M_{n}$ for
$n\geq0$ are in the normal form and satisfy our requirement. But
the operators, constructed from $M_{0}$ applying Proposition \ref{prop:lap-2},
is not necessarily in the normal form. So we construct the operators
for $n<0$ step by step. We construct $M_{-1}=\pa_{x}\pa_{y}+a_{-1}\pa_{x}+c_{-1}$
from $M_{0}$ as follows. Apply Lemma \ref{lem:Lap_-} to $M_{0}$
to obtain 
\[
M_{-}=\pa_{x}\pa_{y}+a_{-}\pa_{x}+b_{-}\pa_{y}+c_{-},
\]
where
\[
a_{-}=a_{0},\;b_{-}=-\pa_{x}\log k_{0},\;c_{-}=c_{0}+\pa_{x}a_{0}-a_{0}\,\pa_{x}\log k_{0}.
\]
 Then applying Lemma \ref{lem:norm-1}, we take $M_{-}$ to the normal
form
\[
M_{-1}=\pa_{x}\pa_{y}+a_{-1}\pa_{x}+c_{-1},
\]
where, using $F=\log k_{0}$, the coefficients are given by
\begin{align*}
a_{-1} & =a_{-}+\pa_{y}\log k_{0}=a_{0}+F_{y},\\
c_{-1} & =c_{-}+(a_{-})F_{x}+(b_{-})F_{y}+F_{x,y}+F_{x}F_{y}=c_{0}+\pa_{x}a_{0}+F_{x,y}.
\end{align*}
We should check that $M_{0}$ is obtained from $M_{-1}$ by the process
in Lemma \ref{lem:Lap_+}. This is easily checked as follows. Let
us denote the operator obtained from $M_{-1}$ as $M_{0}'=\pa_{x}\pa_{y}+a'\pa_{x}+c'$
by the process in Lemma \ref{lem:Lap_+}. Then $a'$ and $c'$ are
obtained from $M_{-1}$ as
\begin{align*}
a' & =a_{-1}-\pa_{y}\log h_{-1}=a_{0}+F_{y}-\pa_{y}\log h_{-1}=a_{0},\\
c' & =c_{-1}-\pa_{x}a_{-1}=c_{0}+\pa_{x}a_{0}+F_{x,y}-\pa_{x}(a_{0}+F_{y})=c_{0}.
\end{align*}
 Here we used $h_{-1}=k_{0}.$ Repeating this construction successively,
we obtain the sequence $M_{0}\to M_{-1}\to M_{-2}\to\cdots$ of the
normal form which satisfy our requirement that $M_{n+1}$ is obtained
from $M_{n}$ by the process in Lemma \ref{lem:Lap_+} for any $n\leq-1$.
Thus we have proved the following.
\begin{prop}
\label{prop:lap-3}From the given $M_{0}=\pa_{x}\pa_{y}+a_{0}\pa_{x}+c_{0}$,
we can construct the sequence of hyperbolic operators of the normal
form
\[
M_{n}=\pa_{x}\pa_{y}+a_{n}\pa_{x}+c_{n},\quad n\in\Z
\]
 such that $M_{n+1}$ is obtained from $M_{n}$ by 
\begin{equation}
a_{n+1}=a_{n}-\pa_{y}\log h_{n},\quad c_{n+1}=c_{n}-\pa_{x}a_{n}\label{eq:laplace-7}
\end{equation}
under the condition that the invariant $h_{n}$ of $M_{n}$ is not
zero for any $n$.
\end{prop}

The sequence $\{M_{n}\}_{n\in\Z}$ obtained in Proposition \ref{prop:lap-3}
is also called the Laplace sequence.

\subsection{\label{subsec:2DTE-laplace}2dTE arising from the Laplace sequence}

In this section, changing the notation used in Section \ref{subsec:The-sequence-normal},
we write the Laplace sequence $\{M_{n}\}_{n\in\Z}$ of the normal
form as
\begin{equation}
M_{n}=\pa_{x}\pa_{y}+s_{n+1}\pa_{x}+r_{n}.\label{eq:toda-4}
\end{equation}
From (\ref{eq:laplace-7}) of Proposition \ref{prop:lap-3}, we have
the recurrence relation for the pair $(s_{n+1},r_{n})$:
\begin{equation}
\begin{cases}
\pa_{x}s_{n+1}=r_{n}-r_{n+1},\\
\pa_{y}\log r_{n}=s_{n}-s_{n+1}.
\end{cases}\label{eq:toda-3}
\end{equation}
Conversely, the following result is known and is easily shown.
\begin{prop}
\label{prop:Toda-1}If $\{(s_{n+1},r_{n})\}_{n\in\Z}$ satisfies (\ref{eq:toda-3}),
then the sequence $\{M_{n}\}_{n\in\Z}$ defined by (\ref{eq:toda-4})
is the Laplace sequence.
\end{prop}

We mainly consider the 2-dimensional Toda-Hirota equation (2dTHE):

\begin{equation}
\pa_{x}\pa_{y}\log\tau_{n}=\frac{\tau_{n+1}\tau_{n-1}}{\tau_{n}^{2}},\quad n\in\Z.\label{eq:toda-1}
\end{equation}

The following result gives the link between 2dTHE and 2dTE, which
is well known and is easily verified.
\begin{prop}
\label{prop:Toda-2}Let $\{\tau_{n}(x,y)\}_{n\in\Z}$ satisfy the
equation (\ref{eq:toda-1}) and let $(s_{n+1},r_{n})$ be defined
by
\begin{equation}
s_{n+1}=\pa_{y}\log\left(\frac{\tau_{n}}{\tau_{n+1}}\right),\quad r_{n}=\pa_{x}\pa_{y}\log\tau_{n},\label{eq:toda-1-1}
\end{equation}
 then $\{(s_{n+1},r_{n})\}_{n\in\Z}$ gives a solution of (\ref{eq:toda-3}),
and $\{r_{n}\}_{n\in\Z}$ satisfies the 2dTE (\ref{eq:toda-2}).
\end{prop}

\subsection{\label{subsec:Darboux-transformation}B\"acklund transformation }

When a solution $\{t_{n}\}_{n\in\Z}$ of the 2dTHE is given, we can
construct a new solution of 2dTHE as explained in the following. Proposition
\ref{prop:Toda-1} tells us that $\{(s_{n+1},r_{n})\}_{n\in\Z}$,
defined by (\ref{eq:toda-1-1}) taking $t_{n}$ as $\tau_{n}$, gives
a Laplace sequence $\{M_{n}\}_{n\in\Z}$ of the form

\[
M_{n}=\pa_{x}\pa_{y}+s_{n+1}\pa_{x}+r_{n}.
\]
Let $\W$ be a simply connected domain where $M_{n}$ are holomorphically
defined, and let $\cS(n)$ be the space of holomorphic solutions of
$M_{n}u=0$ in $\W$. Let us give the differential operators $H_{n}$
and $B_{n}$ which give linear maps
\[
H_{n}:\cS(n)\to\cS(n+1),\quad B_{n}:\cS(n)\to\cS(n-1)
\]
satisfying $B_{n+1}H_{n}=1$ and $H_{n-1}B_{n}=1$ on $\cS(n)$. Let
$C\in\C$ be a quatitiy independent of $x,y$, but may dependent on
some parameters. Define
\begin{equation}
H_{n}=C(\pa_{y}+s_{n+1}),\quad B_{n}=-C^{-1}r_{n}^{-1}\pa_{x}.\label{eq:darboux-1}
\end{equation}
Then this $H_{n}$ and $B_{n}$ satisfies our requirement as is seen
from (\ref{eq:laplace-3-4}) and the construction of the operators
in Proposition \ref{prop:lap-3}. 
\begin{prop}
\label{prop:Backlund-1}Assume that the invariants $r_{n}(=-h_{n-1})$
are nonzero for any $n\in\Z$. Then, for any $n$, $H_{n}$ and $B_{n}$
above define the linear isomorphisms
\[
H_{n}:\cS(n)\to\cS(n+1),\quad B_{n}:\cS(n)\to\cS(n-1).
\]
 
\end{prop}

If we are given a solution $u_{0}$ of $M_{0}u=0$, we can construct
$\{u_{n}\}_{n\in\Z}$ such that $u_{n}\in\cS(n)$ by $u_{n+1}=H_{n}u_{n}$
and $u_{n-1}=B_{n}u_{n}$. The following is the important result to
establish our main result which assert that the Gelfand HGF gives
a solution to the 2dTHE. 

\begin{prop}
\label{prop:Backlund-2}Suppose that $\{t_{n}\}_{n\in\Z}$ is a solution
of 2dTHE (\ref{eq:toda-1}) from which the Laplace sequence $\{M_{n}\}_{n\in\Z}$
is constructed. Given $\{u_{n}\}_{n\in\Z}$ such that $u_{n}\in\cS(n)$
and satisfies $u_{n+1}=H_{n}u_{n}$ and $u_{n-1}=B_{n}u_{n}$. Define
$\{\tau_{n}\}_{n\in\Z}$ by $\tau_{n}=t_{n}u_{n}.$ Then $\{\tau_{n}\}_{n\in\Z}$
gives a solution of the 2dTHE (\ref{eq:toda-1}).
\end{prop}

\begin{proof}
By definition, we have 
\[
\partial_{x}\partial_{y}\log\tau_{n}=\partial_{x}\partial_{y}\log t_{n}+\partial_{x}\partial_{y}\log u_{n}=r_{n}+\partial_{x}\partial_{y}\log u_{n}.
\]
For this $u_{n}$ we show 
\begin{equation}
\partial_{x}\partial_{y}\log u_{n}=\frac{r_{n}u_{n+1}u_{n-1}}{u_{n}^{2}}-r_{n}.\label{eq:Back-1}
\end{equation}
Noting $\partial_{x}\partial_{y}\log u_{n}=\partial_{x}\partial_{y}u_{n}/u_{n}-\partial_{x}u_{n}\cdot\partial_{y}u_{n}/u_{n}^{2}$
and using $C^{-1}H_{n}=\partial_{y}+s_{n+1},CB_{n}=-r_{n}^{-1}\partial_{x}$,
we compute 
\begin{align*}
\partial_{x}\partial_{y}u_{n} & =-s_{n+1}\partial_{x}u_{n}-r_{n}u_{n}=-s_{n+1}\cdot(-Cr_{n}B_{n}u_{n})-r_{n}u_{n}\\
 & =Cr_{n}s_{n+1}u_{n-1}-r_{n}u_{n},\\
\partial_{x}u_{n}\cdot\partial_{y}u_{n} & =(-Cr_{n}B_{n}u_{n})\cdot(C^{-1}H_{n}u_{n}-s_{n+1}u_{n})\\
 & =-r_{n}u_{n+1}u_{n-1}+Cr_{n}s_{n+1}u_{n-1}u_{n}.
\end{align*}
Hence we have (\ref{eq:Back-1}). It follows that 
\begin{align*}
\partial_{x}\partial_{y}\log\tau_{n} & =r_{n}\frac{u_{n+1}u_{n-1}}{u_{n}^{2}}=\partial_{x}\partial_{y}\log t_{n}\cdot\frac{u_{n+1}u_{n-1}}{u_{n}^{2}}\\
 & =\frac{t_{n+1}t_{n-1}}{t_{n}^{2}}\cdot\frac{u_{n+1}u_{n-1}}{u_{n}^{2}}=\frac{\tau_{n+1}\tau_{n-1}}{\tau_{n}^{2}}.
\end{align*}
\end{proof}

\section{Gelfand HGF of type $\protect\lm$}

\subsection{\label{subsec:Definition-of-HGF}Definition of the Gelfand HGF}

Let $N>2$ be a positive integer and let $G=\GL N$ be the complex
general linear group consisting of $N\times N$ invertible matrices.
The Lie algebra of $G$ is denoted by $\gee=\gl(N)$. Given a partition
of $N$, $\lm=(n_{1},n_{2},\dots,n_{\ell})$, namely a tuple of non-increasing
integers $n_{1}\geq n_{2}\geq\cdots\geq n_{\ell}>0$ with $|\lm|:=n_{1}+n_{2}+\cdots+n_{\ell}=N$.
To the partition $\lm$, we associate a maximal abelian subgroup $H_{\lm}=J(n_{1})\times\cdots\times J(n_{\ell})$
of $G$, where 
\[
J(n)=\left\{ h=\left(\begin{array}{cccc}
h_{0} & h_{1} & \dots & h_{n-1}\\
 & \ddots & \ddots & \vdots\\
 &  & \ddots & h_{1}\\
 &  &  & h_{0}
\end{array}\right)\mid h_{0},\dots,h_{n-1}\in\C\right\} \subset\GL n
\]
and $(h^{(1)},\dots,h^{(\ell)})\in J(n_{1})\times\cdots\times J(n_{\ell})$
is identified with the block diagonal matrix $\diag(h^{(1)},\dots,h^{(\ell)})\in G$.
With the shift matrix $\La=(\de_{i+1,j})_{0\leq i,j<n}$ of size $n$,
$h\in J(n)$ is expressed as $h=\sum_{0\leq k<n}h_{k}\La^{k}$. The
Lie algebras of $H_{\lm}$ and $J(n)$ are denoted as $\ha_{\lm}$
and $\fj(n)$, respectively. Then $\ha_{\lm}=\fj(n_{1})\oplus\cdots\oplus\fj(n_{\ell})$
is an abelian Lie subalgebra of $\gee$.

The Gelfand HGF of type $\lm$, which we consider in this paper, is
defined as a Radon transform of a character $\chi$ of the universal
covering group $\tilde{H}_{\lm}$ of $H_{\lm}$. Let $t=(t_{0},t_{1})$
be the homogeneous coordinate of the 1-dimensional complex projective
space $\Ps^{1}$. The Radon transform is, roughly speaking, to substitute
linear polynomials of $t$ in $\chi$, and then to integrate it on
an appropriate 1-chain in $\Ps^{1}$. So we give the explicit form
of a character $\chi:\tilde{H}_{\lm}\to\cbatu$. 

Let $x=(x_{0},x_{1},x_{2},\cdots)$ be variables and let $T$ be an
indeterminate. Define the functions $\te_{k}(x),\;k\geq0,$ by 
\[
\log(x_{0}+x_{1}T+x_{2}T^{2}+\cdots)=\log x_{0}+\log\left(1+\sum_{k\geq1}\frac{x_{k}}{x_{0}}T^{k}\right)=\sum_{k\geq0}\te_{k}(x)T^{k}.
\]
The right hand side is obtained by using $\log(1+X)=\sum_{k\geq1}\frac{(-1)^{k+1}}{k}X^{k}$.
Then $\te_{0}(x)=\log x_{0}$ and 
\[
\theta_{k}(x)=\sum\frac{(-1)^{m+1}}{m}\left(\frac{x_{i_{1}}}{x_{0}}\right)\cdots\left(\frac{x_{i_{m}}}{x_{0}}\right)\quad(k\geq1),
\]
where the sum is taken over the tuples $(i_{1},\dots,i_{m})$ with
$i_{1},\dots,i_{m}\geq1$ and $i_{1}+\cdots+i_{m}=k$. In particular,
\[
\te_{1}(x)=\frac{x_{1}}{x_{0}},\te_{2}(x)=\frac{x_{2}}{x_{0}}-\frac{1}{2}\left(\frac{x_{1}}{x_{0}}\right)^{2},\te_{3}(x)=\frac{x_{3}}{x_{0}}-\left(\frac{x_{1}}{x_{0}}\right)\left(\frac{x_{2}}{x_{0}}\right)+\frac{1}{3}\left(\frac{x_{1}}{x_{0}}\right)^{3}.
\]

\begin{lem}
The correspondence
\[
h=\sum_{0\leq k<n}h_{k}\La^{k}\mapsto(h_{0},\te_{1}(h),\dots,\te_{n-1}(h))
\]
gives a group isomorphism $J(n)\to\cbatu\times\C^{n-1}$, where $\C^{n-1}$
is an additive group with the addition of vectors.
\end{lem}

By this lemma, we see that the universal covering group $\tilde{J}(n)$
of $J(n)$ is isomorphic to $\tilde{\cbatu}\times\C^{n-1}$. This
fact allows us to describe the characters of $\tilde{J}(n)$ and $\tilde{H}_{\lm}$.
\begin{lem}
For a character $\chi$ of $\tilde{J}(n)$, there exist $\al=(\al_{0},\dots,\al_{n-1})\in\C^{n}$
such that 
\[
\chi(h)=h_{0}^{\al_{0}}\exp\left(\sum_{1\leq k<n}\al_{k}\te_{k}(h)\right),\quad h=\sum_{0\leq k<n}h_{k}\La^{k}.
\]
 This character will be denoted as $\chi_{n}(\cdot;\al)$.
\end{lem}

\begin{lem}
Any character $\chi:\tilde{H}_{\lm}\to\cbatu$ is given by
\[
\chi(h;\al)=\prod_{1\leq i\leq\ell}\chi_{n_{i}}(h^{(i)};\al^{(i)}),\quad h=\diag(h^{(1)},\dots,h^{(\ell)}),\quad h^{(i)}\in\tilde{J}(n_{i})
\]
for some $\al=(\al^{(1)},\dots,\al^{(\ell)})\in\C^{N},\quad\al^{(i)}=(\al_{0}^{(i)},\al_{1}^{(i)},\dots,\al_{n_{i}-1}^{(i)})\in\C^{n_{i}}$.
\end{lem}

For a character $\chi(\cdot;\al)$ of $\tilde{H}_{\lm}$, we assume
the condition
\begin{equation}
\al_{n_{i}-1}^{(i)}\begin{cases}
\neq0 & \text{if}\;n_{i}\geq2,\\
\notin\Z & \text{if}\;n_{i}=1,
\end{cases}\quad\al_{0}^{(1)}+\cdots+\al_{0}^{(\ell)}=-2.\label{eq:cond-parameter}
\end{equation}
The vector $\al$ is considered as an element of $\ha_{\lm}^{*}:=\mathrm{Hom}_{\C}(\ha_{\lm},\C)$,
where $\al^{(i)}\in\fj(n_{i})^{*}$ is determined by $\al^{(i)}(\La^{k})=\al_{k}^{(i)}$.
To consider the Radon transform of a character of $\tilde{H}_{\lm}$,
we prepare $N$ linear polynomials of $t=(t_{0},t_{1})$ by specifying
the coefficients of them. Let $\matt{2,N}$ denote the set of $2\times N$
complex matrices of rank $2$. Write $z\in\matt{2,N}$ as 
\begin{equation}
z=(z^{(1)},\dots,z^{(\ell)}),\;z^{(i)}=(z_{0}^{(i)},\dots,z_{n_{i}-1}^{(i)}),\;z_{k}^{(i)}=\left(\begin{array}{c}
z_{0,k}^{(i)}\\
z_{1,k}^{(i)}
\end{array}\right)\in\C^{2}.\label{eq:gel-0}
\end{equation}
Then the generic stratum $Z_{\lm}\subset\matt{2,N}$ is defined by
\[
Z_{\lm}=\left\{ z\in\mat{2,N}\mid\det(z_{0}^{(i)},z_{0}^{(j)})\neq0\;(\forall i\neq j),\quad\det(z_{0}^{(i)},z_{1}^{(i)})\neq0\;(\text{for }n_{i}\geq2)\right\} .
\]
 For $z\in Z_{\lm}$, we define $N$ linear polynomials by
\[
tz=(tz^{(1)},\dots,tz^{(\ell)}),\quad tz^{(i)}=(tz_{0}^{(i)},\dots,tz_{n_{i}-1}^{(i)}),\quad tz_{k}^{(i)}=t_{0}z_{0,k}^{(i)}+t_{1}z_{1,k}^{(i)}.
\]
We use the convention to identify $tz^{(i)}$ with $\sum_{0\leq k<n_{i}}(tz_{k}^{(i)})\La^{k}\in\tilde{J}(n_{i})$
and to regard $tz$ as an element of $\tilde{H}_{\lm}$.
\begin{defn}
The Gelfand HGF of type $\lm$ is the function on $Z_{\lm}$ defined
by
\begin{equation}
F(z,\al;C)=\int_{C(z)}\chi(tz;\al)\cdot\tau,\quad\tau=t_{0}dt_{1}-t_{1}dt_{0},\label{eq:gel-1}
\end{equation}
where $C=\{C(z)\}$ is an appropriate $1$-cycle in the $t$-space
$\Ps^{1}$ representing an element of the homology group of locally
finite 1-chains with coefficients in the local system and with the
family of supports defined by $\chi(tz;\al)\cdot\tau$. See \cite{Pham}
for this kind of homology group.
\end{defn}

Let us write the above integral using the affine coordinate $s:=t_{1}/t_{0}$
in the chart $\{[t]\in\Ps^{1}\mid t_{0}\neq0\}$, where $[t]$ denotes
the point of $\Ps^{1}$ with the homogeneous coordinate $t$. Note
that $\tau=t_{0}^{2}d(t_{1}/t_{0})$. Thus putting $\vec{s}=(1,s)$
and taking account of (\ref{eq:cond-parameter}), we have
\begin{equation}
F(z,\al;C)=\int_{C}\chi(\vec{s}z;\al)ds.\label{eq:gel-2}
\end{equation}

The following property is important to understand the classical HGF
family as the Gelfand HGF of type $\lm$, which will be discussed
in Section \ref{subsec:Classical-HGFs}. Define the action of $\GL 2\times H_{\lm}$
on $\matt{2,N}$ by 
\begin{equation}
\GL 2\times\matt{2,N}\times H_{\lm}\ni(g,z,h)\mapsto gzh\in\matt{2,N}.\label{eq:gel-3}
\end{equation}
It is easy to see that this induces the action on $Z_{\lm}$.
\begin{prop}
\label{prop:covariance-1}The following identities hold.
\begin{align}
F(zh,\al;C) & =\chi(h;\al)F(z,\al;C),\qquad h\in H_{\lm},\label{eq:cova-1}\\
F(gz,\al;C) & =(\det g)^{-1}F(z,\al;C'),\qquad g\in\GL 2,\label{eq:cova-2}
\end{align}

\noindent where $C'=\{C'(z)\}$ is obtained from $C(z)$ by the projective
transformation $\Ps^{1}\ni[t]\mapsto[s]:=[tg]\in\Ps^{1}$.
\end{prop}

\begin{prop}
\label{prop:Gelfand-eq} The Gelfand HGF satisfies the differential
equations:

\begin{align}
\square_{p,q}F=(\pa_{0,p}\pa_{1,q}-\pa_{1,p}\pa_{0,q})F & =0,\quad1\leq p,q\leq N,\label{eq:gel-eq1}\\
L_{X}F:=\left(\Tr(\tr(zX)\pa)-\al(X)\right)F & =0,\quad X\in\ha_{\lm},\label{eq:gel-eq-2}\\
M_{Y}F:=\left(\Tr(\tr(Yz)\pa)+\Tr(Y)\right)F & =0,\quad Y\in\gl(2),\label{eq:gel-eq3}
\end{align}
where $\pa_{i,p}=\pa/\pa z_{i,p}$. 
\end{prop}

The system of differential equations in the above Proposition is called
the \emph{Gelfand hypergeometric system} (Gelfand HGS). Note that
(\ref{eq:gel-eq-2}) and (\ref{eq:gel-eq3}) are infinitesimal forms
of the property (\ref{eq:cova-1}) and (\ref{eq:cova-2}) of Proposition
\ref{prop:covariance-1}, and (\ref{eq:gel-eq1}) is the system which
characterizes the image of Radon transform, respectively.

\subsection{\label{subsec:Contiguity}Contiguity relation for the Gelfand HGF}

Let us recall how the contiguity relations for the Gelfand HGF are
described. See also \cite{Kimura-H-T}. We use the simultaneous generalized
eigenspace decomposition of $\gee$ with respect to $\{\ad_{X}\in\mathrm{End}(\gee)\mid X\in\ha_{\lm}\}$,
where $\ad_{X}(Y):=[X,Y]:=XY-YX$. Since $\ha_{\lm}=\fj(n_{1})\oplus\cdots\oplus\fj(n_{\ell})$
is an abelian Lie subalgebra of $\gee$, $\{\ad_{X}\mid X\in\ha_{\lm}\}$
is a commuting family of Lie algebra homomorphisms. Then $\gee$ is
decomposed into the simultaneous generalized eigenspaces with respect
to this commuting family of endomorphisms:
\begin{equation}
\gee=\ha_{\lm}\oplus\bigoplus_{\al\in\De}\gee_{\al},\quad\De=\{\ep^{(i)}-\ep^{(j)}\mid1\leq i\neq j\leq\ell\}.\label{eq:cont-gel-2}
\end{equation}
Here, $\ep^{(i)}$ is the element of the dual space $\ha_{\lm}^{*}$
of $\ha_{\lm}$ defined by $\ep^{(i)}(X)=X_{0}^{(i)}$ for 
\begin{equation}
X=\diag(X^{(1)},\dots,X^{(\ell)})\in\ha_{\lm},\quad X^{(i)}=\sum_{0\leq k<n_{i}}X_{k}^{(i)}\La^{k}\in\fj(n_{i}),\label{eq:cont-gel-1}
\end{equation}
$\gee_{\ep^{(i)}-\ep^{(j)}}=\left\{ Y\in\gee\mid\left(\ad_{X}-(\ep^{(i)}-\ep^{(j)})(X)\right)^{m}Y=0\;\text{for \ensuremath{\forall X\in\ha_{\lm}}, \ensuremath{\exists}\ensuremath{m\in}\ensuremath{\mathbb{N}}}\right\} $
is the generalized eigenspace for $\ep^{(i)}-\ep^{(j)}$, $\De$ is
the set of roots and its element is called a root. It is seen that
each generalized eigenspace $\gee_{\ep^{(i)}-\ep^{(j)}}$ contains
$1$-dimensional eigenspace with a basis vector $E_{\ep^{(i)}-\ep^{(j)}}$
which is the matrix unit given as follows. According as the partition
$\lm=(n_{1},\dots,n_{\ell})$, we express $Y\in\gee$ in block-wise
as 
\[
Y=\left(\begin{array}{ccc}
Y^{(1,1)} & \cdots & Y^{(1,\ell)}\\
\vdots &  & \vdots\\
Y^{(\ell,1)} & \cdots & Y^{(\ell,\ell)}
\end{array}\right)
\]
with the $(p,q)$-block $Y^{(p,q)}\in\mat{n_{p},n_{q}}$. Then $E_{\ep^{(i)}-\ep^{(j)}}$
is the matrix unit in $\gee$ whose only nonzero element $1$ locates
at the upper right corner of $(i,j)$-block. An element of this eigenspace
is called a root vector corresponding to the root $\ep^{(i)}-\ep^{(j)}$.
For each root $\ep^{(i)}-\ep^{(j)}\in\De$, the contiguity relation
for the Gelfand HGF of type $\lm$ is provided by the infinitesimal
action of the $1$-paramter subgroup $s\mapsto\exp(sE_{\ep^{(i)}-\ep^{(j)}})$
on $Z_{\lm}$. Let $Y$ be a root vector in the generalized root space
decomposition. Then the space $Z_{\lm}$ is preserved by $z\mapsto z\exp(sY)$
and the $1$-parameter subgroup $\{\exp(sY)\}_{s\in\C}$ acts on the
functions on it. Then we obtain the first order differential operators
on $Z_{\lm}$ as its infinitesimal action. Namely, for a holomorphic
function $f$ on $Z_{\lm}$, define the differential operator $L_{Y}$
by 
\begin{equation}
(L_{Y}f)(z):=\frac{d}{ds}f(z\exp(sY))|_{s=0}.\label{eq:cont-gel-3}
\end{equation}
 When $Y=E_{\ep^{(i)}-\ep^{(j)}}$, the operator $L_{Y}$ will be
denoted by $L_{\ep^{(i)}-\ep^{(j)}}$ which is given by 
\[
L_{\ep^{(i)}-\ep^{(j)}}=z_{0,0}^{(i)}\pa_{0,n_{j}-1}^{(j)}+z_{1,0}^{(i)}\pa_{1,n_{j}-1}^{(j)}
\]
in the notation (\ref{eq:gel-0}), where $\pa_{a,n_{j}-1}^{(j)}:=\pa/\pa z_{a,n_{j}-1}^{(j)}$.
\begin{prop}
\label{prop:cont-gel-3}\cite{Kimura-H-T} Let $F$ be the Gelfand
HGF of type $\lm$. Then the contiguity relation is given by
\begin{equation}
L_{\ep^{(i)}-\ep^{(j)}}F(z,\al)=\al_{n_{j}-1}^{(j)}F(z,\al+\ep^{(i)}-\ep^{(j)}),\label{eq:cont-gel-4}
\end{equation}
where $\al\mapsto\al+\ep^{(i)}-\ep^{(j)}$ implies, for $\al=(\al^{(1)},\dots,\al^{(\ell)}),\al^{(k)}=(\al_{0}^{(k)},\dots,\al_{n_{k}-1}^{(k)})$,
the change $\al_{0}^{(i)}\mapsto\al_{0}^{(i)}+1,\al_{0}^{(j)}\mapsto\al_{0}^{(j)}-1$
with the other entries unchanged.
\end{prop}

\section{Action of contiguity operators to the solution space }

In this section, we focus on the system consisting of the equations
\[
\square_{p,q}u=\det\left(\begin{array}{cc}
\pa_{0,p} & \pa_{0,q}\\
\pa_{1,p} & \pa_{1,q}
\end{array}\right)u=0,\quad1\leq p\ne q\leq N,
\]
which forms a main part of the Gelfand HGS. Let $\ideal$ be the left
ideal generated by $\{\square_{p,q}\}$ in the ring of differential
operators with polynomial coefficients on $Z_{\lm}$. Let $Sol(\ideal)$
be the sheaf of holomorphic solutions of the system on $Z_{\lm}.$ 

\subsection{Invariance of $\protect\ideal$ by the action of $H_{\protect\lm}$ }

We show that the ideal $\ideal$ is invariant by the action of $H_{\lm}$
on $Z_{\lm}$ given by $z\mapsto zh$. We adopt the notation used
in (\ref{eq:gel-0}), namely, 
\[
z=(z^{(1)},\dots,z^{(\ell)}),\;z^{(j)}=(z_{0}^{(j)},\dots,z_{n_{j}-1}^{(j)}),\;z_{k}^{(j)}=\left(\begin{array}{c}
z_{0,k}^{(j)}\\
z_{1,k}^{(j)}
\end{array}\right)\in\C^{2}.
\]
Let $h=\diag(h^{(1)},\dots,h^{(\ell)})$, $h^{(j)}\in J(n_{j})$.
Then the action $Z_{\lm}\curvearrowleft H_{\lm}$ is given by $z^{(j)}\mapsto z^{(j)}h^{(j)},1\leq j\leq\ell$.
Accordingly, put 
\[
\pa=(\pa^{(1)},\dots,\pa^{(\ell)}),\;\pa^{(j)}=(\pa_{0}^{(j)},\dots,\pa_{n_{j}-1}^{(j)}),\;\pa_{k}^{(j)}=\left(\begin{array}{c}
\pa_{0,k}^{(j)}\\
\pa_{1,k}^{(j)}
\end{array}\right).
\]
For $z\mapsto w:=zh$, we put $D_{a,k}^{(j)}=\pa/\pa w_{a,k}^{(j)}$
and 
\[
D=(D^{(1)},\dots,D^{(\ell)}),\;D^{(j)}=(D_{0}^{(j)},\dots,D_{n_{j}-1}^{(j)}),\;D_{k}^{(j)}=\left(\begin{array}{c}
D_{0,k}^{(j)}\\
D_{1,k}^{(j)}
\end{array}\right).
\]
Put also
\begin{align*}
\square_{p,q}^{(i,j)} & =\det\left(\pa_{p}^{(i)},\pa_{q}^{(j)}\right)=\pa_{0,p}^{(i)}\pa_{1,q}^{(j)}-\pa_{1,p}^{(i)}\pa_{0,q}^{(j)},\\
\tilde{\square}_{p,q}^{(i,j)} & =\det\left(D_{p}^{(i)},D_{q}^{(j)}\right)=D_{0,p}^{(i)}D_{1,q}^{(j)}-D_{1,p}^{(i)}D_{0,q}^{(j)}.
\end{align*}
 
\begin{prop}
With the notation above, there hold
\begin{align*}
\square_{p,q}^{(i,j)} & =\sum_{p\leq a<n_{i}}\sum_{q\leq b<n_{j}}\tilde{\square}_{a,b}^{(i,j)}h_{a-p}^{(i)}h_{b-q}^{(j)},\\
\tilde{\square}_{p,q}^{(i,j)} & =\sum_{p\leq a<n_{i}}\sum_{q\leq b<n_{j}}\square_{a,b}^{(i,j)}(h^{-1}){}_{a-p}^{(i)}(h^{-1}){}_{b-q}^{(j)}.
\end{align*}
Hence the ideal $\ideal$ is invariant by the action $Z_{\lm}\curvearrowleft H_{\lm}$. 
\end{prop}

\begin{proof}
Since the transformation rule for $\pa\mapsto D$ is contravariant
to that for $z\mapsto w=zh$, we have $\pa=D\cdot\tr h$. Namely $\pa^{(i)}=D^{(i)}\cdot\tr h^{(i)}$
for any $i$, which implies 
\[
\pa_{p}^{(i)}=\sum_{p\leq a<n_{i}}D_{a}^{(i)}h_{a-p}^{(i)}.
\]
Then 
\begin{align*}
\square_{p,q}^{(i,j)} & =\det\left(\pa_{p}^{(i)},\pa_{q}^{(j)}\right)=\det\left(\sum_{p\leq a<n_{i}}D_{a}^{(i)}h_{a-p}^{(i)},\sum_{q\leq b<n_{j}}D_{b}^{(j)}h_{b-q}^{(j)}\right)\\
 & =\sum_{p\leq a<n_{i}}\sum_{q\leq b<n_{j}}\det\left(D_{a}^{(i)},D_{b}^{(j)}\right)h_{a-p}^{(i)}h_{b-q}^{(j)}\\
 & =\sum_{p\leq a<n_{i}}\sum_{q\leq b<n_{j}}\tilde{\square}_{a,b}^{(i,j)}h_{a-p}^{(i)}h_{b-q}^{(j)}.
\end{align*}
\end{proof}
\begin{cor}
If $u(z)\in Sol(\ideal)$, then $u(zh)\in Sol(\ideal)$ for any $h\in H_{\lm}$.
\end{cor}

\subsection{Action of contiguity operators on $Sol(\protect\ideal)$ }

We study the effect of the contiguity operators $L_{\ep^{(i)}-\ep^{(j)}}$
when it is applied to an element of $Sol(\ideal)$. Recall that $L_{\ep^{(i)}-\ep^{(j)}}$
corresponds to the matrix unit $E_{\ep^{(i)}-\ep^{(j)}}$ and is given
by
\[
L_{\ep^{(i)}-\ep^{(j)}}=z_{0,0}^{(i)}\pa_{0,n_{j}-1}^{(j)}+z_{1,0}^{(i)}\pa_{1,n_{j}-1}^{(j)}=(z_{0,0}^{(i)},z_{1,0}^{(i)})\left(\begin{array}{c}
\pa_{0,n_{j}-1}^{(j)}\\
\pa_{1,n_{j}-1}^{(j)}
\end{array}\right)
\]
and 
\[
\square_{a,b}^{(p,q)}=\det(\pa_{a}^{(p)},\pa_{b}^{(q)})=\pa_{0,a}^{(p)}\pa_{1,b}^{(q)}-\pa_{1,a}^{(p)}\pa_{0,b}^{(q)}.
\]

\begin{lem}
\label{lem:Act-1} The correspondence $u\mapsto L_{\ep^{(i)}-\ep^{(j)}}\cdot u$
gives a homomorphism $Sol(\ideal)\to Sol(\ideal)$.
\end{lem}

\begin{proof}
Firstly we compute the commutator of $L_{\ep^{(i)}-\ep^{(j)}}$ and
$\square_{a,b}^{(p,q)}$. (1) In the case $p,q\neq i$, it is easy
to see that $[L_{\ep^{(i)}-\ep^{(j)}},\square_{a,b}^{(p,q)}]=0$ holds
for any $a,b$. (2) In the case $p=i$ or $q=i$. We may assume $p=i$;
(i) when $a\neq0$, then we have $[L_{\ep^{(i)}-\ep^{(j)}},\square_{a,b}^{(i,q)}]=0$
for any $b$ and $q$; (ii) when $a=0$, then we have 
\begin{align*}
[L_{\ep^{(i)}-\ep^{(j)}},\square_{0,b}^{(i,q)}] & =[z_{0,0}^{(i)}\pa_{0,n_{j}-1}^{(j)}+z_{1,0}^{(i)}\pa_{1,n_{j}-1}^{(j)},\pa_{0,0}^{(i)}\pa_{1,b}^{(q)}-\pa_{1,0}^{(i)}\pa_{0,b}^{(q)}]\\
 & =[z_{0,0}^{(i)}\pa_{0,n_{j}-1}^{(j)},\pa_{0,0}^{(i)}\pa_{1,b}^{(q)}]-[z_{1,0}^{(i)}\pa_{1,n_{j}-1}^{(j)},\pa_{1,0}^{(i)}\pa_{0,b}^{(q)}]\\
 & =[z_{0,0}^{(i)},\pa_{0,0}^{(i)}]\pa_{1,b}^{(q)}\pa_{0,n_{j}-1}^{(j)}-[z_{1,0}^{(i)},\pa_{1,0}^{(i)}]\pa_{0,b}^{(q)}\pa_{1,n_{j}-1}^{(j)}\\
 & =-\pa_{1,b}^{(q)}\pa_{0,n_{j}-1}^{(j)}+\pa_{0,b}^{(q)}\pa_{1,n_{j}-1}^{(j)}\\
 & =\square_{b,n_{j}-1}^{(q,j)}.
\end{align*}
Suppose $u\in Sol(\ideal)$, put $v=L_{\ep^{(i)}-\ep^{(j)}}u$, and
we show that $v\in Sol(\ideal)$, namely $\square_{a,b}^{(p,q)}v=0$
for any $p,q,a,b$. In the cases (1) and (2-i), we have 
\[
\square_{a,b}^{(p,q)}v=\square_{a,b}^{(p,q)}L_{\ep^{(i)}-\ep^{(j)}}u=L_{\ep^{(i)}-\ep^{(j)}}\square_{a,b}^{(p,q)}u=0
\]
 and we are done. In the case (2-ii), we have 

\[
\square_{0,b}^{(i,q)}v=\square_{0,b}^{(i,q)}L_{\ep^{(i)}-\ep^{(j)}}u=(L_{\ep^{(i)}-\ep^{(j)}}\square_{0,b}^{(i,q)}-\square_{b,n_{j}-1}^{(q,j)})u=0
\]
and we are done. 
\end{proof}
Put 
\[
\cS(\al):=\{u\in Sol(\ideal)\mid u(zh)=u(z)\chi(h;\al),\;h\in H_{\lm}\}.
\]

\begin{prop}
The correspondence $u\mapsto L_{\ep^{(i)}-\ep^{(j)}}u$ gives a homomophism
$\cS(\al)\to\cS(\al+\ep^{(i)}-\ep^{(j)})$.
\end{prop}

\begin{proof}
Note that the correspondence $Y\mapsto L_{Y}$ defines a Lie algebra
homomorphism $L:\gl(N)\to\mathscr{X}(Z_{\lm})$, where $\mathscr{X}(Z_{\lm})$
is the set of holomorphic vector fields on $Z_{\lm}$. Put $Y=E_{\ep^{(i)}-\ep^{(j)}}$.
Since $Y$ is a root vector corresponding to the root $\ep^{(i)}-\ep^{(j)}$
in the generalized root space decomposition, we have 
\[
\ad_{X}Y:=[X,Y]=(\ep^{(i)}(X)-\ep^{(j)}(X))Y,\quad\forall X\in\ha_{\lm}.
\]
Applying the Lie algebra homomorphism above, we have
\[
[L_{X},L_{Y}]=(\ep^{(i)}(X)-\ep^{(j)}(X))L_{Y}
\]
which is written as 
\begin{equation}
\left(L_{X}-(\ep^{(i)}(X)-\ep^{(j)}(X))\right)L_{Y}=L_{Y}L_{X}.\label{eq:act-1}
\end{equation}
Take $u\in\cS(\al)$. Note that we have $L_{X}u=\al(X)u$, which is
the infinitesimal form of $u(zh)=u(z)\chi(h;\al)$. Then applying
the both sides of (\ref{eq:act-1}) to $u$, we have 
\[
\left(L_{X}-(\al(X)+\ep^{(i)}(X)-\ep^{(j)}(X))\right)L_{Y}u=0.
\]
 This implies that $v=L_{Y}u$ satisfies
\[
v(zh)=v(z)\chi(h;\al+\ep^{(i)}-\ep^{(j)}),\quad\forall h\in H_{\lm}.
\]
Since $v\in Sol(\ideal)$ by virtue of Lemma \ref{lem:Act-1}, we
see that $v\in\cS(\al+\ep^{(i)}-\ep^{(j)})$. 
\end{proof}

\section{$SL(2)$ action on $Sol(\protect\ideal)$ and $\protect\cS(\protect\al)$}

Let $f(z)$ be a function on $Z_{\lm}$. For $g\in SL(2)$, define
$(g^{*}f)(z):=f(g^{-1}z)$.
\begin{lem}
\label{lem:Act-2}For $u\in Sol(\ideal)$ and $g\in SL(2)$, we have
$g^{*}u\in Sol(\ideal)$. Hence the correspondence $u\mapsto g^{*}u$
gives an isomorphism $g^{*}:Sol(\ideal)\to Sol(\ideal)$.
\end{lem}

\begin{proof}
Here we express $z\in Z_{\lm}$ as 
\[
z=(z_{1},\dots,z_{N}),\quad z_{p}=\tr(z_{0,p},z_{1,p}).
\]
Accordingly we use the notation $\pa_{i,p}=\pa/\pa z_{i,p}$ and $\pa_{p}=\tr(\pa_{0,p},\pa_{1,p})$.
We show that operator $\square_{p,q}$ is invariant by the left action
$SL(2)\curvearrowright Z_{\lm}$. Take $g\in SL(2)$ and consider
the change of variable $z\mapsto w=g^{-1}z$. Put $D_{i,p}=\pa/\pa w_{i,p}$
and $\pa=(\pa_{i,p}),D=(D_{i,p})$. We see that $\pa=\tr g^{-1}D$,
namely, $\pa_{p}=\tr g^{-1}D_{p}.$ It follows that 
\begin{align*}
\square_{p,q} & =\det(\pa_{p},\pa_{q})=\det(\tr g^{-1}D_{p},\tr g^{-1}D_{q})\\
 & =\det(\tr g^{-1})\det(\vec{D}_{p},\vec{D}_{q})=\det(\vec{D}_{p},\vec{D}_{q})=:\tilde{\square}_{p,q}.
\end{align*}
As a consequence, we see that if $u\in Sol(\ideal)$, then $g^{*}u\in Sol(\ideal)$.
\end{proof}
\begin{prop}
For $u\in\cS(\al)$ and $g\in SL(2)$, we have $g^{*}u\in\cS(\al)$.
Hence the correspondence $u\mapsto g^{*}u$ gives a linear isomorphism
$g^{*}:\cS(\al)\to\cS(\al)$.
\end{prop}

\begin{proof}
Suppose $u\in\cS(\al)$. Then $u\in Sol(\ideal)$ and by virtue of
Lemma \ref{lem:Act-2}, $g^{*}u\in Sol(\ideal)$. Moreover, for any
$h\in H_{\lm}$, we have 
\[
(g^{*}u)(zh)=u(g^{-1}(zh))=u((g^{-1}z)h)=u(g^{-1}z)\chi(h;\al)=(g^{*}u)(z)\chi(h;\al).
\]
This implies $g^{*}u\in\cS(\al)$.
\end{proof}

\section{\label{sec:Restr-slice}Restriction to the slice}

\subsection{Contiguity operator on the slice}

Consider the slice $X\subset Z_{\lm}$ defined by 
\[
X=\left\{ \bx=(\bx^{(1)},\dots,\bx^{(\ell)})\in Z_{\lm}\mid\bx^{(j)}=\left(\begin{array}{cccc}
x_{0}^{(j)} & x_{1}^{(j)} & \dots & x_{n_{j}-1}^{(j)}\\
1 & 0 & \dots & 0
\end{array}\right)\;\text{for \ensuremath{\forall j}}\right\} .
\]
The condition for $\bx$ to belong to $Z_{\lm}$ is $x_{0}^{(i)}\neq x_{0}^{(j)}\;(i\neq j),x_{1}^{(i)}\neq0\;(\forall i)$.
We consider the restriction $u(\bx;\al)$ of $u(z;\al)\in\cS(\al)$
to the slice $X$ and denote by $\cS^{res}(\al)$ the set of restriction
of elements of $\cS(\al)$ which obtained as the pullback of $\cS(\al)$
by the inclusion $\iota:X\to Z_{\lm}$. We shall find the contiguity
operators and the hyperbolic operators of the form (\ref{eq:laplace-1})
acting on this restriction. Firstly we explain how $Z_{\lm}$ and
$X$ are related. 

Let $z\in Z_{\lm}$ be written as
\[
z=(z^{(1)},\dots,z^{(\ell)}),\;z^{(j)}=(z_{0}^{(j)},\dots,z_{n_{j}-1}^{(j)})\in\matt{2,n_{j}},\;z_{k}^{(j)}=\tr(z_{0,k}^{(j)},z_{1,k}^{(j)}).
\]
 For $i=0,1$, the $i$-th row vector of $z^{(j)}$ is denoted by
$\vec{z}_{i}^{(j)}:=(z_{i,0}^{(j)},\dots,z_{i,n_{j}-1}^{(j)})$. 
\begin{lem}
For $z=(z^{(1)},\dots,z^{(\ell)})\in Z_{\lm}$ above, assume $z_{1,0}^{(j)}\neq0\quad(1\leq j\leq\ell)$.
Then there exists a unique $h\in H_{\lm}$ such that $zh=\bx\in X$.
In this case, $h=\diag(h^{(1)},\dots,h^{(\ell)})$ and $\bx=(\bx^{(1)},\dots,\bx^{(\ell)})$
are given by 
\begin{equation}
h^{(j)}=\left(\sum_{0\leq k<n_{j}}z_{1,k}^{(j)}\La^{k}\right)^{-1},\;(x_{0}^{(j)},\dots,x_{n_{j}-1}^{(j)})=\vec{z}_{0}^{(j)}\left(\sum_{0\leq k<n_{j}}z_{1,k}^{(j)}\La^{k}\right)^{-1}.\label{eq:redu-0}
\end{equation}
\end{lem}

\begin{proof}
The condition $zh=\bx\in X$ can be written as $z^{(j)}h^{(j)}=\bx^{(j)},1\leq j\leq\ell$.
Comparing the $1$-th row of $z^{(j)}h^{(j)}=\bx^{(j)}$, we have
$\vec{z}_{1}^{(j)}h^{(j)}=(1,0,\dots,0)$, which can be written as
$\left(\sum_{0\leq k<n_{j}}z_{1,k}^{(j)}\La^{k}\right)h^{(j)}=1_{n_{j}}$.
Since $z_{1,0}^{(j)}\neq0$ by the assumption, $\sum_{0\leq k<n_{j}}z_{1,k}^{(j)}\La^{k}$
is an invertible matrix and we have $h^{(j)}=\left(\sum_{0\leq k<n_{j}}z_{1,k}^{(j)}\La^{k}\right)^{-1}$.
Comparing the $0$-th row of $\bx^{(j)}=z^{(j)}h^{(j)}$, we have
\begin{equation}
(x_{0}^{(j)},\dots,x_{n_{j}-1}^{(j)})=\vec{z}_{0}^{(j)}\left(\sum_{0\leq k<n_{j}}z_{1,k}^{(j)}\La^{k}\right)^{-1}.\label{eq:redu-1}
\end{equation}
\end{proof}
Suppose $\bx\in X$ is related to $z\in Z_{\lm}$ by $zh=\bx$, $h\in H_{\lm}$.
Then
\[
u(z;\al)=u(\bx h^{-1};\al)=u(\bx;\al)\chi(h^{-1};\al).
\]
We derive the operators acting on $u(\bx;\al)$ from the contiguity
operator for $u(z;\al)$. Recall that the contiguity operator for
the root $\ep^{(i)}-\ep^{(j)}$ is 
\[
L_{\ep^{(i)}-\ep^{(j)}}=\tr z_{0}^{(i)}\pa_{n_{j}-1}^{(j)}=z_{0,0}^{(i)}\pa_{0,n_{j}-1}^{(j)}+z_{1,0}^{(i)}\pa_{1,n_{j}-1}^{(j)}.
\]
Put $D_{k}^{(j)}:=\pa/\pa x_{k}^{(j)}$. We must relate the derivatives
with respect to $z$ to those with respect to $x$. To simplify the
description, we use the convention:
\[
n_{j}\to n,\quad z_{a,0}^{(j)}\to z_{a,0},\quad\pa_{a,n_{j}-1}^{(j)}\to\pa_{a,n-1}\quad(a=0,1)
\]
and $x_{a}^{(j)}\to x_{a}$.
\begin{lem}
\label{lem:Redu-1}When $\pa_{a,n-1}$ is applied to the function
of $x$, it acts as 
\[
\pa_{0,n-1}=\frac{1}{z_{1,0}}D_{n-1},\quad\pa_{1,n-1}=-\frac{z_{0,0}}{z_{1,0}^{2}}D_{n-1},
\]
where $D_{n-1}=\pa/\pa x_{n-1}$.
\end{lem}

\begin{proof}
We can see that, among the variables $x_{0},\dots,x_{n-1}$, only
$x_{n-1}$ depends on $z_{0,n-1}$ and/or $z_{1,n-1}$. To see this,
define the rational functions $\psi_{0},\dots,\psi_{n-1}$ of $y=(y_{0},\dots,y_{n-1})$
by the generating function
\begin{equation}
(y_{0}+y_{1}T+\cdots+y_{n-1}T^{n-1})^{-1}=\sum_{k\geq0}\psi_{k}(y)T^{k};\label{eq:redu-2}
\end{equation}
for example we have
\[
\psi_{0}(y)=\frac{1}{y_{0}},\;\psi_{1}(y)=-\frac{1}{y_{0}}\left(\frac{y_{1}}{y_{0}}\right),\;\psi_{2}(y)=\frac{1}{y_{0}}\left(-\frac{y_{2}}{y_{0}}+\left(\frac{y_{1}}{y_{0}}\right)^{2}\right).
\]
Then (\ref{eq:redu-1}) can be written as 
\begin{equation}
(x_{0},\dots,x_{n-1})=(z_{0,0},\dots,z_{0,n-1})\left(\sum_{0\leq k<n}\psi_{k}(\vec{z}_{1})\La^{k}\right).\label{eq:redu-3}
\end{equation}
Differentiate (\ref{eq:redu-3}) with respect to $z_{0,n-1}$ to obtain
\[
\frac{\pa}{\pa z_{0,n-1}}(x_{0},\dots,x_{n-1})=(0,\dots,0,1)\left(\sum_{0\leq k<n}\psi_{k}(\vec{z}_{1})\La^{k}\right)=(0,\dots,0,\psi_{0}(\vec{z}_{1})).
\]
It follows that $x_{0},\dots,x_{n-2}$ are independent of $z_{0,n-1}$
and $\pa x_{n-1}/\pa z_{0,n-1}=\psi_{0}(\vec{z}_{1})=1/z_{1,0}$.
Similarly, differentiating (\ref{eq:redu-3}) with respect to $z_{1,n-1}$
we have 
\begin{equation}
\frac{\pa}{\pa z_{1,n-1}}(x_{0},\dots,x_{n-1})=(z_{0,0},\dots,z_{0,n-1})\left(\sum_{0\leq k<n}\frac{\pa\psi_{k}(\vec{z}_{1})}{\pa z_{1,n-1}}\La^{k}\right).\label{eq:redu-4}
\end{equation}
To obtain the expression of $\pa x_{n-1}/\pa z_{1,n-1}$, differentiate
the both sides of (\ref{eq:redu-2}) with respect to $y_{n-1}$ and
get
\[
-(y_{0}+y_{1}T+\cdots+y_{n-1}T^{n-1})^{-2}T^{n-1}=\sum_{k\geq0}\frac{\pa\psi_{k}(y)}{\pa y_{n-1}}T^{k}.
\]
 Equating the coefficients of $T^{n-1}$ of the both sides, we have
\[
\frac{\pa\psi_{k}(y)}{\pa y_{n-1}}=\begin{cases}
0, & 0\leq k<n-1,\\
-\frac{1}{y_{0}^{2}}, & k=n-1.
\end{cases}
\]
Then (\ref{eq:redu-4}) is written as 
\[
\frac{\pa}{\pa z_{1,n-1}}(x_{0},\dots,x_{n-1})=(z_{0,0},\dots,z_{0,n-1})\left(-\frac{1}{z_{1,0}^{2}}\right)\La^{n-1}.
\]
Thus we see that $x_{0},\dots,x_{n-2}$ are independent of $z_{1,n-1}$
and $\pa x_{n-1}/\pa z_{1,n-1}=-z_{0,0}/z_{1,0}^{2}$. Hence we have
\[
\pa_{0,n-1}=\frac{\pa x_{n-1}}{\pa z_{0,n-1}}D_{n-1}=\frac{1}{z_{1,0}}D_{n-1},\quad\pa_{1,n-1}=\frac{\pa x_{n-1}}{\pa z_{1,n-1}}D_{n-1}=-\frac{z_{0,0}}{z_{1,0}^{2}}D_{n-1}.
\]
 
\end{proof}
Now we can obtain the form of the contiguity operator when it acts
on the restriction $u(\bx;\al)$.
\begin{prop}
\label{prop:Redu-2} The contiguity operator $L_{\ep^{(i)}-\ep^{(j)}}$
for $u(z;\al)\in\cS(\al)$ induces the operator $\cL_{\ep^{(i)}-\ep^{(j)}}(\al)$
for $u(\bx;\al)$ given by 
\[
\cL_{\ep^{(i)}-\ep^{(j)}}(\al)=(x_{0}^{(i)}-x_{0}^{(j)})D_{n_{j}-1}^{(j)}+\al_{n_{j}-1}^{(j)}.
\]
\end{prop}

\begin{proof}
Let $z$ and $\bx$ be related as $\bx=zh$, which is written as $\bx^{(k)}=z^{(k)}h^{(k)}$,
$h^{(k)}=\left(\sum_{0\leq a<n_{k}}z_{1,a}^{(k)}\La^{a}\right)^{-1}$.
Taking into account this expression, we sometimes use the convention
that the row vector $\vec{z}_{1}^{(k)}$ is regarded as an element
$\sum_{0\leq a<n_{k}}z_{1,a}^{(k)}\La^{a}\in J(n_{k})$. Then
\[
u(z;\al)=u(\bx h^{-1};\al)=u(\bx;\al)\chi(h^{-1};\al)=u(\bx;\al)\prod_{1\leq k\leq\ell}\chi_{n_{k}}(\vec{z}_{1}^{(k)};\al^{(k)}).
\]
Then 
\begin{align}
 & L_{\ep^{(i)}-\ep^{(j)}}u(z;\al)\nonumber \\
 & =L_{\ep^{(i)}-\ep^{(j)}}\left(u(\bx;\al)\chi(h^{-1};\al)\right)\nonumber \\
 & =L_{\ep^{(i)}-\ep^{(j)}}u(\bx;\al)\cdot\chi(h^{-1};\al)+u(\bx;\al)\cdot L_{\ep^{(i)}-\ep^{(j)}}\chi(h^{-1};\al).\label{eq:redu-5}
\end{align}
By Lemma \ref{lem:Redu-1}, when $L_{\ep^{(i)}-\ep^{(j)}}$ acts on
a function of $x$, it behaves as 
\begin{align}
L_{\ep^{(i)}-\ep^{(j)}} & =z_{0,0}^{(i)}\pa_{0,n_{j}-1}^{(j)}+z_{1,0}^{(i)}\pa_{1,n_{j}-1}^{(j)}\nonumber \\
 & =z_{0,0}^{(i)}\frac{1}{z_{1,0}^{(j)}}D_{n_{j}-1}^{(j)}+z_{1,0}^{(i)}\left(-\frac{z_{0,0}^{(j)}}{(z_{1,0}^{(j)})^{2}}\right)D_{n_{j}-1}^{(j)}\nonumber \\
 & =\frac{z_{1,0}^{(i)}}{z_{1,0}^{(j)}}(x_{0}^{(i)}-x_{0}^{(j)})D_{n_{j}-1}^{(j)}.\label{eq:redu-6}
\end{align}
For the second term of (\ref{eq:redu-5}), we see that
\begin{equation}
L_{\ep^{(i)}-\ep^{(j)}}\chi(h^{-1};\al)=\al_{n_{j}-1}^{(j)}\frac{z_{1,0}^{(i)}}{z_{1,0}^{(j)}}\chi(h^{-1};\al)\label{eq:redu-7}
\end{equation}
holds. In fact, noting $\chi(h^{-1};\al)=\prod_{1\leq k\leq\ell}\chi_{n_{k}}(\vec{z}_{1}^{(k)};\al^{(k)})$,
we have 
\[
L_{\ep^{(i)}-\ep^{(j)}}\chi(h^{-1};\al)=\chi(h^{-1};\al)\cdot L_{\ep^{(i)}-\ep^{(j)}}\log\chi_{n_{j}}(\vec{z}_{1}^{(j)};\al^{(j)}).
\]
Since $\log\chi_{n_{j}}(\vec{z}_{1}^{(j)};\al^{(j)})=\al_{0}^{(j)}\log z_{1,0}^{(j)}+\sum_{1\leq a<n_{j}}\al_{a}^{(j)}\te_{a}(\vec{z}_{1}^{(j)})$,
the term depending on $z_{1,n_{j}-1}^{(j)}$ is 
\[
\al_{n_{j}-1}^{(j)}\te_{n_{j}-1}(\vec{z}_{1}^{(j)})=\al_{n_{j}-1}^{(j)}\left(\frac{z_{1,n_{j}-1}^{(j)}}{z_{1,0}^{(j)}}+(\text{terms independent of \ensuremath{z_{1,n_{j}-1}^{(j)}}})\right),
\]
and we have $L_{\ep^{(i)}-\ep^{(j)}}\log\chi_{n_{j}}(\vec{z}_{1}^{(j)};\al^{(j)})=\al_{n_{j}-1}^{(j)}z_{1,0}^{(i)}/z_{1,0}^{(j)}$.
Hence (\ref{eq:redu-7}) holds. Combining (\ref{eq:redu-5}), (\ref{eq:redu-6})
and (\ref{eq:redu-7}), we see that 
\[
L_{\ep^{(i)}-\ep^{(j)}}u(z;\al)=\frac{z_{1,0}^{(i)}}{z_{1,0}^{(j)}}\chi(h^{-1};\al)\left\{ (x_{0}^{(i)}-x_{0}^{(j)})D_{n_{j}-1}^{(j)}+\al_{n_{j}-1}^{(j)}\right\} u(\bx;\al)
\]
holds. This proves the proposition.
\end{proof}
We show that a hyperbolic operator of the form (\ref{eq:laplace-1})
is obtained from $\square_{n_{i}-1,n_{j}-1}^{(i,j)}$ by similar reduction
done in Proposition \ref{prop:Redu-2}.
\begin{prop}
\label{prop:Redu-3}The operator $\square_{n_{i}-1,n_{j}-1}^{(i,j)}$
for $u(z;\al)\in\cS(\al)$ induces the operator $\tilde{M}^{(i,j)}(\al)$
for $u(\bx;\al)$ given by 
\begin{equation}
\tilde{M}^{(i,j)}(\al)=(x_{0}^{(i)}-x_{0}^{(j)})D_{n_{i}-1}^{(i)}D_{n_{j}-1}^{(j)}+\al_{n_{j}-1}^{(j)}D_{n_{i}-1}^{(i)}-\al_{n_{i}-1}^{(i)}D_{n_{j}-1}^{(j)}.\label{eq:redu-8}
\end{equation}
\end{prop}

\begin{proof}
We adopt the setting in the proof of Proposition \ref{prop:Redu-2}.
For the sake of brevity, put $u(z)=u(z;\al)$, $v(x)=u(\bx;\al)$
and $\chi:=\chi(h^{-1};\al)=\prod_{1\leq k\leq\ell}\chi_{n_{k}}(\vec{z}_{1}^{(k)};\al^{(k)})$.
Noting that $\chi$ depends only on $\vec{z}_{1}^{(k)}$$(1\leq k\leq\ell)$,
applying $\square:=\square_{n_{i}-1,n_{j}-1}^{(i,j)}=\pa_{0,n_{i}-1}^{(i)}\pa_{1,n_{j}-1}^{(j)}-\pa_{1,n_{i}-1}^{(i)}\pa_{0,n_{j}-1}^{(j)}$
to $u=\chi\cdot v$, we have 
\begin{equation}
\square u=\chi\cdot\square v+(\pa_{1,n_{j}-1}^{(j)}\chi)\pa_{0,n_{i}-1}^{(i)}v-(\pa_{1,n_{i}-1}^{(i)}\chi)\pa_{0,n_{j}-1}^{(j)}v.\label{eq:redu-9}
\end{equation}
For the second and the third term, we know that 
\[
\pa_{1,n_{i}-1}^{(i)}\chi=\frac{\al_{n_{i}-1}^{(i)}}{z_{1,0}^{(i)}}\chi,\quad\pa_{1,n_{j}-1}^{(j)}\chi=\frac{\al_{n_{j}-1}^{(j)}}{z_{1,0}^{(j)}}\chi
\]
as in the proof of Proposition \ref{prop:Redu-2}, and by virtue of
Lemma \ref{lem:Redu-1}, we have 
\begin{align*}
(\pa_{1,n_{j}-1}^{(j)}\chi)\pa_{0,n_{i}-1}^{(i)}v & =\frac{\al_{n_{j}-1}^{(j)}}{z_{1,0}^{(j)}}\chi\cdot\frac{1}{z_{1,0}^{(i)}}D_{n_{i}-1}^{(i)}v=\chi\cdot\frac{1}{z_{1,0}^{(i)}z_{1,0}^{(j)}}\al_{n_{j}-1}^{(j)}D_{n_{i}-1}^{(i)}v,\\
(\pa_{1,n_{i}-1}^{(i)}\chi)\pa_{0,n_{j}-1}^{(j)}v & =\frac{\al_{n_{i}-1}^{(i)}}{z_{1,0}^{(i)}}\chi\cdot\frac{1}{z_{1,0}^{(j)}}D_{n_{j}-1}^{(j)}v=\chi\cdot\frac{1}{z_{1,0}^{(i)}z_{1,0}^{(j)}}\al_{n_{i}-1}^{(i)}D_{n_{j}-1}^{(j)}v.
\end{align*}
For the first term of (\ref{eq:redu-9}), we have
\begin{align*}
\square v & =\pa_{0,n_{i}-1}^{(i)}\cdot\left(-\frac{z_{0,0}^{(j)}}{(z_{1,0}^{(j)})^{2}}D_{n_{j}-1}^{(j)}v\right)-\pa_{0,n_{j}-1}^{(j)}\cdot\left(-\frac{z_{0,0}^{(i)}}{(z_{1,0}^{(i)})^{2}}D_{n_{i}-1}^{(i)}v\right)\\
 & =-\frac{1}{z_{1,0}^{(i)}}\frac{z_{0,0}^{(j)}}{(z_{1,0}^{(j)})^{2}}D_{n_{i}-1}^{(i)}D_{n_{j}-1}^{(j)}v+\frac{1}{z_{1,0}^{(j)}}\frac{z_{0,0}^{(i)}}{(z_{1,0}^{(i)})^{2}}D_{n_{i}-1}^{(i)}D_{n_{j}-1}^{(j)}v\\
 & =\frac{1}{z_{1,0}^{(i)}z_{1,0}^{(j)}}(x_{0}^{(i)}-x_{0}^{(j)})D_{n_{i}-1}^{(i)}D_{n_{j}-1}^{(j)}v.
\end{align*}
 Putting these in (\ref{eq:redu-9}), we obtain 
\[
\square u=\chi\cdot\frac{1}{z_{1,0}^{(i)}z_{1,0}^{(j)}}\left((x_{0}^{(i)}-x_{0}^{(j)})D_{n_{i}-1}^{(i)}D_{n_{j}-1}^{(j)}+\al_{n_{j}-1}^{(j)}D_{n_{i}-1}^{(i)}-\al_{n_{i}-1}^{(i)}D_{n_{j}-1}^{(j)}\right)v.
\]
Thus from the equation $\square u=0$, we obtain the equation $\tilde{M}^{(i,j)}(\al)v=0$
. 
\end{proof}
The following proposition says the relation between the operator $\tilde{M}^{(i,j)}(\al)$
and the contiguity operators $\{\cL_{\ep^{(i)}-\ep^{(j)}}(\al)\}$,
which is shown by direct computation. 
\begin{prop}
\label{prop:Redu-4}We have 
\begin{multline*}
\cL_{\ep^{(j)}-\ep^{(i)}}(\al+\ep^{(i)}-\ep^{(j)})\cdot\cL_{\ep^{(i)}-\ep^{(j)}}(\al)\\
=-(x_{0}^{(i)}-x_{0}^{(j)})\tilde{M}^{(i,j)}(\al)+(\al_{n_{i}-1}^{(i)}+\de_{n_{i},1})\al_{n_{j}-1}^{(j)}.
\end{multline*}
\end{prop}

\subsection{Expression of $SL(2)$ action on the slice $X$}

We consider the action of $SL(2)$ on the slice $X$ and on $\cS^{res}(\al)$.
Take $g\in SL(2)$ and $\bx=(\bx^{(1)},\dots,\bx^{(\ell)})\in X$.
Then, finding $h\in H_{\lm}$ such that ${\bf y}:=g\bx h\in X$, we
define the action by the correspondence $\bx\mapsto{\bf y}$. This
${\bf y}$ is denoted by $g_{*}\bx$. Explicitly, for $\bx=(\bx^{(1)},\dots,\bx^{(\ell)})$
and ${\bf y}=({\bf y}^{(1)},\dots,{\bf y}^{(\ell)})$, the relation
${\bf y}=g\bx h$ is written as ${\bf y}^{(j)}=g\bx^{(j)}h^{(j)}\;(1\leq j\leq\ell)$,
so it is sufficient to determine $\bx^{(j)}\mapsto{\bf y}^{(j)}$
for each $j$. We consider a fixed $j$ and use again the convention:
to omit the superscript $(j)$ and to write $n_{j}$ as $n$. So we
put 
\[
\bx=\left(\begin{array}{cccc}
x_{0} & x_{1} & \dots & x_{n-1}\\
1 & 0 & \dots & 0
\end{array}\right),\;{\bf y}=\left(\begin{array}{cccc}
y_{0} & y_{1} & \dots & y_{n-1}\\
1 & 0 & \dots & 0
\end{array}\right)
\]
and find $h=\sum_{0\leq k<n}h_{k}\La^{k}\in J(n)$ such that ${\bf y}=g\bx h$
has the similar form as $\bx$. Note that

\[
g\bx=\left(\begin{array}{c}
(\vec{g\bx})_{0}\\
(\vec{g\bx})_{1}
\end{array}\right)=\left(\begin{array}{cccc}
ax_{0}+b & ax_{1} & \dots & ax_{n-1}\\
cx_{0}+d & cx_{1} & \dots & cx_{n-1}
\end{array}\right)\;\text{for}\;g=\left(\begin{array}{cc}
a & b\\
c & d
\end{array}\right)\in SL(2).
\]
Hence $h$ should be defined as $h=(\vec{g\bx})_{1}^{-1}:=\left((cx_{0}+d)+\sum_{1\leq k<n}cx_{k}\La^{k}\right)^{-1}$
and ${\bf y}$ be given by
\[
(y_{0},y_{1},\dots,y_{n-1})=(\vec{g\bx})_{0}(\vec{g\bx})_{1}^{-1}=(\vec{g\bx})_{0}\sum_{0\leq k<n}\psi_{k}((\vec{g\bx})_{1})\La^{k},
\]
more explicitly
\[
y_{k}=(ax_{0}+b)\psi_{k}((\vec{g\bx})_{1})+ax_{1}\psi_{k-1}((\vec{g\bx})_{1})+\cdots+ax_{k}\psi_{0}((\vec{g\bx})_{1}),
\]
where $\psi_{k}$ are those defined by (\ref{eq:redu-2}). For example,
we have 
\[
y_{0}=\frac{ax_{0}+b}{cx_{0}+d},\quad y_{1}=-\frac{ax_{0}+b}{cx_{0}+d}\frac{cx_{1}}{cx_{0}+d}+\frac{ax_{1}}{cx_{0}+d}.
\]
 The action of $SL(2)$ on $\cS^{res}(\al)$ is as follows.
\begin{prop}
For $u(\bx)\in\cS^{res}(\al)$ and $g\in SL(2)$, we have
\[
(g^{*}u)(\bx):=\chi((\vec{g\bx})_{1}^{-1};\al)u(g_{*}\bx)\in\cS^{res}(\al).
\]
\end{prop}

\section{\label{sec:Lap-gelfand}Laplace sequence arising from $M^{(i,j)}(\protect\al)$}

In the previous section we obtained the hyperbolic differential operators
\begin{equation}
M^{(i,j)}(\al)=D_{n_{i}-1}^{(i)}D_{n_{j}-1}^{(j)}+\frac{\al_{n_{j}-1}^{(j)}}{x_{0}^{(i)}-x_{0}^{(j)}}D_{n_{i}-1}^{(i)}+\frac{\al_{n_{i}-1}^{(i)}}{x_{0}^{(j)}-x_{0}^{(i)}}D_{n_{j}-1}^{(j)},\label{eq:Lap-1}
\end{equation}
$1\leq i\neq j\leq\ell$, on the slice $X$, which is the result of
reduction of $\square_{n_{i}-1,n_{j}-1}^{(i,j)}$ similar to that
carried out for the contiguity operators $L_{\ep^{(i)}-\ep^{(j)}}$.
We can also obtain the second order differential operators from the
other $\square_{a,b}^{(i,j)}$. See Section \ref{subsec:System-reduction}
for the examples. For the operator $M^{(i,j)}(\al)$, we construct
the Laplace sequence and obtain a particular solution to the 2dTHE
by the help of Proposition \ref{prop:Toda-2}. Since $M^{(i,j)}(\al)$
has the form symmetric with respect to the exchange of the index $i\leftrightarrow j$,
we may assume that $n_{i}\leq n_{j}$ in the following discussion.
The construction is divided in three cases: (i) $n_{i}=n_{j}=1$,
(ii) $n_{i}=1$, $n_{j}\geq2$, (iii) $n_{i},n_{j}\geq2$. 

\subsection{The case $n_{i}=1,n_{j}=1$}

The operator (\ref{eq:Lap-1}) has the form
\begin{equation}
M^{(i,j)}(\al)=D_{0}^{(i)}D_{0}^{(j)}+\frac{\al_{0}^{(j)}}{x_{0}^{(i)}-x_{0}^{(j)}}D_{0}^{(i)}+\frac{\al_{0}^{(i)}}{x_{0}^{(j)}-x_{0}^{(i)}}D_{0}^{(j)}.\label{eq:Lap-2}
\end{equation}
Note that $\al_{0}^{(i)},\al_{0}^{(j)}\notin\Z$ by the assumption
(\ref{eq:cond-parameter}). To simplify the description, we change
the notation:
\[
x_{0}^{(i)}\to x,\;x_{0}^{(j)}\to y,\;D_{0}^{(i)}\to\pa_{x},\;D_{0}^{(j)}\to\pa_{y},\;\al_{0}^{(i)}\to\al,\;\al_{0}^{(j)}\to\be.
\]
Then (\ref{eq:Lap-2}) is written in the form
\begin{equation}
M_{0}(\al,\be):=\pa_{x}\pa_{y}+\frac{\be}{x-y}\pa_{x}+\frac{\al}{y-x}\pa_{y}.\label{eq:Lap-5}
\end{equation}
Note that $\al,\be\notin\Z$. The operator (\ref{eq:Lap-5}) is the
EPD operator. The normal form of (\ref{eq:Lap-5}) in the sense of
Lemma \ref{lem:norm-1} is 

\[
N_{0}(\al,\be)=\pa_{x}\pa_{y}+\frac{\be-\al}{x-y}\pa_{x}+\frac{\al(\be+1)}{(x-y)^{2}},
\]
and the Laplace sequence $\{N_{m}(\al,\be)\}_{m\in\Z}$ is given by
\begin{equation}
N_{m}(\al,\be)=\pa_{x}\pa_{y}+\frac{\be-\al-2m}{x-y}\pa_{x}+\frac{(\al+m)(\be-m+1)}{(x-y)^{2}},\label{eq:Lap-6-1}
\end{equation}
which is the normal form of 
\[
M_{m}(\al,\be):=\pa_{x}\pa_{y}+\frac{\be-m}{x-y}\pa_{x}+\frac{\al+m}{y-x}\pa_{y}.
\]

\begin{lem}
\label{lem: Lap-1}\cite{hiro-kimu} The Laplace sequence (\ref{eq:Lap-6-1})
gives a particular solution to the 2dTHE
\begin{equation}
\pa_{x}\pa_{y}\log t_{m}=\frac{t_{m+1}t_{m-1}}{t_{m}^{2}},\quad m\in\Z\label{eq:Lap-6-2}
\end{equation}
of the form
\[
t_{m}(x,y)=(x-y)^{p(\al,\be;m)}T_{m}(\al,\be),\quad p(\al,\be;m)=(\al+m)(\be-m+1),
\]
where $T_{0}(\al,\be)=1,T_{1}(\al,\be)=A$ with an arbitrary constant
$A$, and 
\begin{equation}
T_{m}(\al,\be)=\begin{cases}
A^{m}\prod_{k=0}^{m-1}\left(\prod_{l=1}^{k}p(\al,\be;l)\right), & m\geq2,\\
A^{m}\prod_{k=1}^{|m|}\left(\prod_{l=-k+1}^{0}p(\al,\be;l)\right), & m\leq-1.
\end{cases}\label{eq:Lap-6-3}
\end{equation}
\end{lem}

Applying the above lemma to the EPD operator (\ref{eq:Lap-2}), we
have the following.
\begin{prop}
\label{prop:Lap-2}For a pair $(i,j)$ such that $n_{i}=n_{j}=1$,
the operator (\ref{eq:Lap-1}) is the EPD operator:
\[
M_{0}(\al):=M^{(i,j)}(\al)=D_{0}^{(i)}D_{0}^{(j)}+\frac{\al_{0}^{(j)}}{x_{0}^{(i)}-x_{0}^{(j)}}D_{0}^{(i)}+\frac{\al_{0}^{(i)}}{x_{0}^{(j)}-x_{0}^{(i)}}D_{0}^{(j)}
\]
 obtained from $\square_{0,0}^{(i,j)}$ by the reduction, whose normal
form in the sense of Lemma \ref{lem:norm-1} is
\[
N_{0}(\al)=D_{0}^{(i)}D_{0}^{(j)}+\frac{\al_{0}^{(j)}-\al_{0}^{(i)}}{x_{0}^{(i)}-x_{0}^{(j)}}D_{0}^{(i)}+\frac{\al_{0}^{(i)}(\al_{0}^{(j)}+1)}{(x_{0}^{(i)}-x_{0}^{(j)})^{2}}.
\]
Then, under the condition $\al_{0}^{(i)},\al_{0}^{(j)}\notin\Z$,
the Laplace sequence $\{N_{m}(\al)\}_{m\in\Z}$ is given by 
\[
N_{m}(\al)=D_{0}^{(i)}D_{0}^{(j)}+\frac{\al_{0}^{(j)}-\al_{0}^{(i)}-2m}{x_{0}^{(i)}-x_{0}^{(j)}}D_{0}^{(i)}+\frac{(\al_{0}^{(i)}+m)(\al_{0}^{(j)}-m+1)}{(x_{0}^{(i)}-x_{0}^{(j)})^{2}},
\]
 which is the normal form of 
\[
M_{m}(\al)=D_{0}^{(i)}D_{0}^{(j)}+\frac{\al_{0}^{(j)}-m}{x_{0}^{(i)}-x_{0}^{(j)}}D_{0}^{(i)}+\frac{\al_{0}^{(i)}+m}{x_{0}^{(j)}-x_{0}^{(i)}}D_{0}^{(j)}.
\]
This Laplace sequence gives a solution to the 2dTHE $D_{0}^{(i)}D_{0}^{(j)}\log t_{m}=t_{m-1}t_{m+1}/t_{m}^{2}$
of the form
\[
t_{m}(\bx;\al)=(x_{0}^{(i)}-x_{0}^{(j)})^{p(\al_{0}^{(i)},\al_{0}^{(j)};m)}T_{m}(\al_{0}^{(i)},\al_{0}^{(j)}),
\]
where $p(\al,\be;m)$ and $T_{m}(\al,\be)$ are those defined in Lemma
\ref{lem: Lap-1}.
\end{prop}

\subsection{The case $n_{i}=1,n_{j}\protect\geq2$}

In this case, the operator (\ref{eq:Lap-1}) has the form

\begin{equation}
M^{(i,j)}(\al)=D_{0}^{(i)}D_{n_{j}-1}^{(j)}+\frac{\al_{n_{j}-1}^{(j)}}{x_{0}^{(i)}-x_{0}^{(j)}}D_{0}^{(i)}+\frac{\al_{0}^{(i)}}{x_{0}^{(j)}-x_{0}^{(i)}}D_{n_{j}-1}^{(j)}.\label{eq:Lap-7}
\end{equation}
Recall that $\al_{0}^{(i)}\notin\Z$ and $\al_{n_{j}-1}^{(j)}\neq0$
by the assumption (\ref{eq:cond-parameter}). For the sake of simplicity
of description, we temporarily change the notation:
\begin{gather*}
x_{0}^{(i)}\to x,\;x_{n_{j}-1}^{(j)}\to y,\;x_{0}^{(j)}\to u,\;D_{0}^{(i)}\to\pa_{x},\;D_{n_{j}-1}^{(j)}\to\pa_{y},\\
\al_{0}^{(i)}\to\al,\;\al_{n_{j}-1}^{(j)}\to\be.
\end{gather*}
Then (\ref{eq:Lap-7}) is written in the form
\begin{equation}
M_{0}(\al,\be):=\pa_{x}\pa_{y}+\frac{\be}{x-u}\pa_{x}+\frac{\al}{u-x}\pa_{y}.\label{eq:Lap-10}
\end{equation}
Note that $\al\notin\Z$, $\be\neq0$. The normal form of (\ref{eq:Lap-10})
in the sense of Lemma \ref{lem:norm-1} is 

\[
N_{0}(\al,\be)=\pa_{x}\pa_{y}+\frac{\be}{x-u}\pa_{x}+\frac{\al\be}{(u-x)^{2}},
\]
and it is easily seen that the associated Laplace sequence $\{N_{m}(\al,\be)\}_{m\in\Z}$
is given by 
\begin{equation}
N_{m}(\al,\be)=\pa_{x}\pa_{y}+\frac{\be}{x-u}\pa_{x}+\frac{(\al+m)\be}{(u-x)^{2}},\label{eq:Lap-10-1}
\end{equation}
which is a normal form of 
\[
M_{m}(\al,\be):=\pa_{x}\pa_{y}+\frac{\be}{x-u}\pa_{x}+\frac{\al+m}{u-x}\pa_{y}.
\]

\begin{lem}
\label{lem:Lap-3}The Laplace sequence (\ref{eq:Lap-10-1}) gives
a particular solution to the 2dTHE (\ref{eq:Lap-6-2}) of the form 

\begin{equation}
t_{m}(x,y)=\exp\left(\frac{(\al+n)\be y}{x-u}\right)(x-u)^{-m(m-1)}\cdot T_{m}(\al,\be),\label{eq:Lap-10-2}
\end{equation}
where $T_{m}(\al,\be)=A^{m}\be^{m(m-1)/2}\prod_{k\in\la0,m\ra}[\al]_{k}$,
$A$ an arbitrary constant, $\la0,m\ra$ being the set on integers
between $0$ and $m$, and 
\[
[\al]_{k}=\begin{cases}
\G(\al+k)/\G(\al+1), & k\geq1,\\
1 & k=0,\\
\G(\al+1)/\G(\al+1+k) & k\leq-1.
\end{cases}
\]
 
\end{lem}

\begin{proof}
Recall that the relation between the Laplace sequence $\{N_{m}\}_{m\in\Z}$
and a solution to the 2dTHE was explained in Section \ref{subsec:2DTE-laplace}.
In the setting of Proposition \ref{prop:Toda-2}, $r_{m}=(\al+m)\be/(x-u)^{2}$
and $t_{m}$ should satisfy $\pa_{x}\pa_{y}\log t_{m}=r_{m}$. Integrating
the both sides with respect to $x$ and $y$, we see that it may be
possible to find $t_{m}$ in the form 
\[
t_{m}=\exp(g_{m})\cdot U_{m}(x),\quad g_{m}=\frac{(\al+m)\be y}{u-x}.
\]
Note that $\exp(g_{m+1})\exp(g_{m-1})/\exp(g_{m})^{2}=1$. So in order
that this $t_{m}$ gives a solution to (\ref{eq:Lap-6-2}), it is
sufficient to determine $U_{m}(x)$ by 
\begin{equation}
\frac{U_{m+1}U_{m-1}}{U_{m}^{2}}=\frac{t_{m+1}t_{m-1}}{t_{m}^{2}}=r_{m}=\frac{(\al+m)\be}{(u-x)^{2}},\quad m\in\Z.\label{eq:Lap-10-3}
\end{equation}
We solve (\ref{eq:Lap-10-3}) under the condition $U_{0}=1,U_{1}=A$
with an arbitrary constant $A$. Consider it for $m\geq1$. Put $V_{m}=U_{m}/U_{m-1}$.
Then $V_{1}=A$ and (\ref{eq:Lap-10-3}) is written as
\begin{equation}
\frac{V_{m+1}}{V_{m}}=\frac{(\al+m)\be}{(u-x)^{2}}.\label{eq:Lap-10-4}
\end{equation}
Then 
\[
\frac{U_{m}}{U_{m-1}}=V_{m}=V_{1}\prod_{1\leq k<m}\frac{(\al+k)\be}{(u-x)^{2}}=A\frac{\be^{m-1}}{(u-x)^{2(m-1)}}\frac{\G(\al+m)}{\G(\al+1)}.
\]
This is solved as 
\[
U_{m}=A^{m}\frac{\be^{m(m-1)/2}}{(u-x)^{m(m-1)}}\prod_{1\leq k\leq m}\frac{\G(\al+k)}{\G(\al+1)}=A^{m}\frac{\be^{m(m-1)/2}}{(u-x)^{m(m-1)}}\prod_{k\in\la0,m\ra}[\al]_{k}
\]
under the condition $U_{0}=1$. Next we consider the case $m\leq-1$.
Putting $m=-p$, write (\ref{eq:Lap-10-3}) as 
\[
\frac{U_{-(p+1)}}{U_{-p}}\slash\frac{U_{-p}}{U_{-(p-1)}}=\frac{(\al-p)\be}{(u-x)^{2}}.
\]
If we put $W_{p}:=U_{-p}/U_{-(p-1)}$, this equation can be written
as $W_{p+1}/W_{p}=(\al-p)\be/(u-x)^{2}$, which is similar to (\ref{eq:Lap-10-4}).
Noting $W_{0}=U_{0}/U_{1}=A^{-1},$ it is solved as
\[
U_{-p}/U_{-(p-1)}=W_{p}=A^{-1}\frac{\be^{p}}{(u-x)^{2p}}\prod_{0\leq l<p}(\al-l),
\]
and hence we have 
\[
U_{-p}=A^{-p}\frac{\be^{p(p+1)/2}}{(u-x)^{p(p+1)}}\prod_{1\leq k\leq p}\frac{\G(\al+1)}{\G(\al-k+1)}=A^{m}\frac{\be^{m(m-1)/2}}{(u-x)^{m(m-1)}}\prod_{k\in\la0,m\ra}[\al]_{k}.
\]
\end{proof}
Applying the above Lemma to the operator (\ref{eq:Lap-7}), we have
the following.
\begin{prop}
\label{prop:Lap-4}For a pair of indices $(i,j)$ such that $n_{i}=1,n_{j}\geq2$,
the operator (\ref{eq:Lap-1}) is the hyperbolic operator
\[
M_{0}(\al):=M^{(i,j)}(\al)=D_{0}^{(i)}D_{n_{j}-1}^{(j)}+\frac{\al_{n_{j}-1}^{(j)}}{x_{0}^{(i)}-x_{0}^{(j)}}D_{0}^{(i)}+\frac{\al_{0}^{(i)}}{x_{0}^{(j)}-x_{0}^{(i)}}D_{n_{j}-1}^{(j)}
\]
 obtained from $\square_{0,n_{j}-1}^{(i,j)}$ by the reduction, and
its normal form is
\[
N_{0}(\al)=D_{0}^{(i)}D_{n_{j}-1}^{(j)}+\frac{\al_{n_{j}-1}^{(j)}}{x_{0}^{(i)}-x_{0}^{(j)}}D_{0}^{(i)}+\frac{\al_{0}^{(i)}\al_{n_{j}-1}^{(j)}}{(x_{0}^{(i)}-x_{0}^{(j)})^{2}}.
\]
Then the associated Laplace sequence $\{N_{m}(\al)\}_{m\in\Z}$ is
given by 
\[
N_{m}(\al)=D_{0}^{(i)}D_{n_{j}-1}^{(j)}+\frac{\al_{n_{j}-1}^{(j)}}{x_{0}^{(i)}-x_{0}^{(j)}}D_{0}^{(i)}+\frac{(\al_{0}^{(i)}+m)\al_{n_{j}-1}^{(j)}}{(x_{0}^{(i)}-x_{0}^{(j)})^{2}},
\]
 which is the normal form of 
\[
M_{m}(\al)=D_{0}^{(i)}D_{n_{j}-1}^{(j)}+\frac{\al_{n_{j}-1}^{(j)}}{x_{0}^{(i)}-x_{0}^{(j)}}D_{0}^{(i)}+\frac{\al_{0}^{(i)}+m}{x_{0}^{(j)}-x_{0}^{(i)}}D_{n_{j}-1}^{(j)}.
\]
This Laplace sequence gives a solution to the 2dTHE $D_{0}^{(i)}D_{n_{j}-1}^{(j)}\log t_{m}=t_{m-1}t_{m+1}/t_{m}^{2}$
of the form
\[
t_{m}(\bx;\al)=\exp\left(\frac{(\al_{0}^{(i)}+m)\al_{n_{j}-1}^{(j)}x_{n_{j}-1}^{(j)}}{x_{0}^{(i)}-x_{0}^{(j)}}\right)(x_{0}^{(i)}-x_{0}^{(j)})^{-m(m+1)}T_{m}(\al_{0}^{(i)},\al_{n_{j}-1}^{(j)})
\]
 where $T_{m}$ is that given in Lemma \ref{lem:Lap-3}. 
\end{prop}

\subsection{The case $n_{i}\protect\geq2,n_{j}\protect\geq2$}

The operator we consider is 
\begin{equation}
M^{(i,j)}(\al)=D_{n_{i}-1}^{(i)}D_{n_{j}-1}^{(j)}+\frac{\al_{n_{j}-1}^{(j)}}{x_{0}^{(i)}-x_{0}^{(j)}}D_{n_{i}-1}^{(i)}+\frac{\al_{n_{i}-1}^{(i)}}{x_{0}^{(j)}-x_{0}^{(i)}}D_{n_{j}-1}^{(j)}.\label{eq:Lap-11}
\end{equation}
Note that $\al_{n_{i}-1}^{(i)},\al_{n_{j}-1}^{(j)}\neq0$ by the assumption
(\ref{eq:cond-parameter}). Noting $n_{i},n_{j}\geq2$, we change
the notation to simplify the description:
\begin{gather*}
x_{n_{i}-1}^{(i)}\to x,\quad x_{n_{j}-1}^{(j)}\to y,\quad x_{0}^{(i)}\to u,\quad x_{0}^{(j)}\to v,\\
D_{n_{i}-1}^{(i)}\to\pa_{x},\quad D_{n_{j}-1}^{(j)}\to\pa_{y},\quad\al_{n_{i}-1}^{(i)}\to\al,\quad\al_{n_{j}-1}^{(j)}\to\be.
\end{gather*}
Then the above operators can be written in the form

\begin{equation}
M_{0}(\al,\be):=\pa_{x}\pa_{y}+\frac{\be}{u-v}\pa_{x}+\frac{\al}{v-u}\pa_{y}.\label{eq:Lap-12}
\end{equation}
The normal form of (\ref{eq:Lap-12}) is 

\[
N_{0}(\al,\be)=\pa_{x}\pa_{y}+\frac{\be}{u-v}\pa_{x}+\frac{\al\be}{(u-v)^{2}}
\]
and its Laplace sequence $\{N_{m}(\al,\be)\}_{m\in\Z}$ is given by
$N_{m}(\al,\be)=N_{0}(\al,\be)$, and evidently it is the normal form
of $M_{m}(\al,\be):=M_{0}(\al,\be)$. Note that $\al,\be\neq0$ by
the assumption.
\begin{lem}
\label{lem:Lap-5}The above Laplace sequence gives a particular solution
to the 2dTHE (\ref{eq:Lap-6-2}) of the form 

\begin{equation}
t_{m}=\exp\left(\frac{\al\be xy}{(u-v)^{2}}\right)(u-v)^{-m(m-1)}\cdot T_{m}(\al,\be),\label{eq:Lap-13}
\end{equation}
where $T_{m}(\al,\be)=A^{m}(\al\be)^{m(m-1)/2}$ and $A$ is an arbitrary
constant.
\end{lem}

\begin{proof}
From the Laplace sequence above, we find $t_{m}$ such that $\pa_{x}\pa_{y}\log t_{m}=r_{m}=\al\be/(u-v)^{2}$
following the recipe of Proposition \ref{prop:Toda-2}. Then we seek
a solution to the 2dTHE in the form $t_{m}=\exp\left(\al\be xy/(u-v)^{2}\right)\cdot U_{m}(u,v)$.
The condition for $U_{m}$ is then given by 
\begin{equation}
\frac{U_{m+1}U_{m-1}}{U_{m}^{2}}=\frac{\al\be}{(u-v)^{2}},\quad m\in\Z.\label{eq:Lap-14}
\end{equation}
We solve this equation under the condition $U_{0}=1,U_{1}=A$ with
an arbitrary constant $A$. For $m\geq1$, put $V_{m}=U_{m}/U_{m-1}$.
Then (\ref{eq:Lap-14}) is written as $V_{m+1}/V_{m}=\al\be/(u-v)^{2}$
and is solved under the condition $V_{1}=U_{1}/U_{0}=A$ as $V_{m}=A(\al\be)^{m-1}/(u-v)^{2(m-1)}$,
and in turn, this $V_{m}$ gives an equation for $U_{m}$:
\[
\frac{U_{m}}{U_{m-1}}=A\frac{(\al\be)^{m-1}}{(u-v)^{2(m-1)}}.
\]
Then 
\[
U_{m}=U_{0}\prod_{k=1}^{m}A\frac{(\al\be)^{k-1}}{(u-v)^{2(k-1)}}=A^{m}\frac{(\al\be)^{m(m-1)/2}}{(u-v)^{m(m-1)}}.
\]
Thus we obtain (\ref{eq:Lap-13}) for $m\geq0$. It is easy to see
that this expression for $t_{m}$ is also valid for $m<0$.
\end{proof}
Applying this lemma to the operator (\ref{eq:Lap-11}), we obtain
the following. 
\begin{prop}
\label{prop:Lap-6}For a pair of indices $(i,j)$ such that $n_{i},n_{j}\geq2$,
the operator (\ref{eq:Lap-1}) is the hyperbolic operator
\[
M_{0}(\al):=M^{(i,j)}(\al)=D_{n_{i}-1}^{(i)}D_{n_{j}-1}^{(j)}+\frac{\al_{n_{j}-1}^{(j)}}{x_{0}^{(i)}-x_{0}^{(j)}}D_{n_{i}-1}^{(i)}+\frac{\al_{n_{i}-1}^{(i)}}{x_{0}^{(j)}-x_{0}^{(i)}}D_{n_{j}-1}^{(j)}
\]
 obtained from $\square_{n_{i}-1,n_{j}-1}^{(i,j)}$ by the reduction,
whose normal form is 
\[
N_{0}(\al)=D_{n_{i}-1}^{(i)}D_{n_{j}-1}^{(j)}+\frac{\al_{n_{j}-1}^{(j)}}{x_{0}^{(i)}-x_{0}^{(j)}}D_{n_{i}-1}^{(i)}+\frac{\al_{n_{i}-1}^{(i)}\al_{n_{j}-1}^{(j)}}{(x_{0}^{(i)}-x_{0}^{(j)})^{2}}.
\]
Then the Laplace sequence $\{N_{m}(\al)\}_{m\in\Z}$ is given by $N_{m}(\al)=N_{0}(\al)$
which is the normal form of $M_{m}(\al):=M_{0}(\al)$. This Laplace
sequence gives a solution to the 2dTHE $D_{n_{i}-1}^{(i)}D_{n_{j}-1}^{(j)}\log t_{m}=t_{m-1}t_{m+1}/t_{m}^{2}$
of the form
\[
t_{m}(\bx;\al)=\exp\left(\frac{\al_{n_{i}-1}^{(i)}\al_{n_{j}-1}^{(j)}x_{n_{i}-1}^{(i)}x_{n_{j}-1}^{(j)}}{(x_{0}^{(i)}-x_{0}^{(j)})^{2}}\right)(x_{0}^{(i)}-x_{0}^{(j)})^{-m(m+1)}T_{m}(\al_{n_{i}-1}^{(i)},\al_{n_{j}-1}^{(j)})
\]
 where $T_{m}$ is that given in Lemma \ref{lem:Lap-5}. 
\end{prop}

\section{\label{sec:HGF-2dTHE}Gelfand HGF as a solution to the 2dTHE}

We construct a solution of the 2dTHE expressed in terms of the Gelfand
HGF of type $\lm$ on $\gras(2,N)$ using Proposition \ref{prop:Backlund-2}.
Recall that $\cS^{res}(\alpha)$ denotes the set of the restrictions
of elements $u(z)\in\cS(\al)$ to $X$. In Proposition \ref{prop:Redu-3},
we showed that the restriction $u(\bx;\al)$ of $u(z;\al)$ to $X$
satisfies the equation 
\[
M^{(i,j)}(\al)v=\left(D_{n_{i}-1}^{(i)}D_{n_{j}-1}^{(j)}+\frac{\al_{n_{j}-1}^{(j)}}{x_{0}^{(i)}-x_{0}^{(j)}}D_{n_{i}-1}^{(i)}+\frac{\al_{n_{i}-1}^{(i)}}{x_{0}^{(j)}-x_{0}^{(i)}}D_{n_{j}-1}^{(j)}\right)v=0
\]
which is obtained from the operator $\square_{n_{i}-1,n_{j}-1}^{(i,j)}$
by the reduction. Note that it corresponds to the root $\ep^{(i)}-\ep^{(j)}$,
and is related to the contiguity operator
\[
\cL_{\ep^{(i)}-\ep^{(j)}}(\al)=(x_{0}^{(i)}-x_{0}^{(j)})D_{n_{j}-1}^{(j)}+\al_{n_{j}-1}^{(j)}
\]
as is stated in Proposition \ref{prop:Redu-4}. We fix $\al\in\ha^{*}$
satisfying the condition:
\begin{equation}
\al_{n_{i}-1}^{(i)}\begin{cases}
\neq0 & \text{if}\;n_{i}\geq2,\\
\notin\Z & \text{if}\;n_{i}=1,
\end{cases}\quad\al_{0}^{(1)}+\cdots+\al_{0}^{(\ell)}=-2.\label{eq:gel-toda-3}
\end{equation}
We also fix an ordered pair $(i,j)$, $1\leq i\neq j\le\ell$, such
that $n_{i}\leq n_{j}$, and consider the Laplace sequence arising
from $M_{0}^{(i,j)}(\al):=M^{(i,j)}(\al)$. For the sake of brevity,
$M^{(i,j)}(\alpha+m(\ep^{(i)}-\ep^{(j)}))$ and $\cS^{res}(\alpha+m(\ep^{(i)}-\ep^{(j)}))$
are denoted as $M_{m}^{(i,j)}(\alpha)$ and $\cS_{m}^{res}(\alpha)$,
respectively. Then 
\[
M_{m}^{(i,j)}(\alpha)=D_{n_{i}-1}^{(i)}D_{n_{j}-1}^{(j)}+\frac{\al_{n_{j}-1}^{(j)}-m\de_{n_{j},1}}{x_{0}^{(i)}-x_{0}^{(j)}}D_{n_{i}-1}^{(i)}+\frac{\al_{n_{i}-1}^{(i)}+m\de_{n_{i},1}}{x_{0}^{(j)}-x_{0}^{(i)}}D_{n_{j}-1}^{(j)}.
\]
Let us give the operators
\[
H_{m}:\cS_{m}^{res}(\alpha)\to\cS_{m+1}^{res}(\alpha),\quad B_{m}:\cS_{m}^{res}(\alpha)\to\cS_{m-1}^{res}(\alpha)
\]
satisfying $B_{m+1}H_{m}=1,\;H_{m-1}B_{m}=1$ on $\cS_{m}^{res}(\alpha)$.
Put $c_{m}^{(i,j)}(\al)=(\al_{n_{i}-1}^{(i)}+m\de_{n_{i},1})(\al_{n_{j}-1}^{(j)}-(m-1)\de_{n_{j},1})$.
Then they are defined by
\begin{align}
H_{m} & =\cL_{\ep^{(i)}-\ep^{(j)}}(\alpha+m(\ep^{(i)}-\ep^{(j)}))\label{eq:gel-toda-1}\\
 & =(x_{0}^{(i)}-x_{0}^{(j)})D_{n_{j}-1}^{(j)}+\al_{n_{j}-1}^{(j)}-m\de_{n_{j},1},\nonumber \\
B_{m} & =\frac{1}{c_{m}^{(i,j)}(\al)}\cL_{\ep^{(j)}-\ep^{(i)}}(\alpha+m(\ep^{(i)}-\ep^{(j)}))\label{eq:gel-toda-2}\\
 & =\frac{1}{c_{m}^{(i,j)}(\al)}\left\{ (x_{0}^{(j)}-x_{0}^{(i)})D_{n_{i}-1}^{(i)}+\al_{n_{i}-1}^{(i)}+m\de_{n_{i},1}\right\} ,\nonumber 
\end{align}
Note that $c_{m}^{(i,j)}(\al)\neq0$ by the condition (\ref{eq:gel-toda-3}).
We know that the normal form of $M_{m}^{(i,j)}(\alpha)$ in the sense
of Lemma \ref{lem:norm-1} is given by 
\[
N_{m}^{(i,j)}(\alpha)=D_{n_{i}-1}^{(i)}D_{n_{j}-1}^{(j)}+\frac{\al_{n_{j}-1}^{(j)}-m\de_{n_{j},1}-\de_{n_{j},1}(\al_{n_{i}-1}^{(i)}+m\de_{n_{i},1})}{x_{0}^{(i)}-x_{0}^{(j)}}D_{n_{i}-1}^{(i)}+\frac{c_{m}^{(i,j)}(\al)}{(x_{0}^{(i)}-x_{0}^{(j)})^{2}}.
\]
Recall that the normal form $N_{m}^{(i,j)}(\alpha)$ is obtained from
$M_{m}^{(i,j)}(\alpha)$ as 
\[
N_{m}^{(i,j)}(\alpha)=(\mathrm{Ad}\,g_{m})M_{m}^{(i,j)}(\alpha):=g_{m}\cdot M_{m}^{(i,j)}(\alpha)\cdot g_{m}^{-1}
\]
with
\begin{equation}
g_{m}(x)=\begin{cases}
(x_{0}^{(i)}-x_{0}^{(j)})^{-(\alpha_{0}^{(i)}+m)}, & n_{i}=1,\\
\exp\left(-\frac{\alpha_{n_{i}-1}^{(i)}x_{n_{i}-1}^{(i)}}{x_{0}^{(i)}-x_{0}^{(j)}}\right), & n_{i}\geq2.
\end{cases}\label{eq:gel-toda-4}
\end{equation}
Thus we have the diagram
\begin{equation}
\begin{CD}M_{m+1}^{(i,j)}(\alpha)@>\mathrm{Ad}\,g_{m+1}>>N_{m+1}^{(i,j)}(\alpha)\\
@AH_{m}AA@AAH_{m}'A\\
M_{m}^{(i,j)}(\alpha)@>\mathrm{Ad}\,g_{m}>>N_{m}^{(i,j)}(\alpha)\\
@VB_{n}VV@VVB_{n}'V\\
M_{m-1}^{(i,j)}(\alpha)@>\mathrm{Ad}\,g_{m-1}>>N_{m-1}^{(i,j)}(\alpha)
\end{CD}\label{eq:gel-toda-5}
\end{equation}
where the vertical arrow $H_{m}$ implies that the operator $M_{m+1}^{(i,j)}(\alpha)$
is determined from $M_{m}^{(i,j)}(\alpha)$ by the change of unknown
$u\mapsto u'=\cL_{\ep^{(i)}-\ep^{(j)}}(\alpha+m(\ep^{(i)}-\ep^{(j)}))u$
for $M_{m}^{(i,j)}(\alpha)u=0$. In this situation, we can determine
the operator $H_{m}'$ so that the above diagram is commutative. We
can show that $H_{m}'$ is determined as $H_{m}'=g_{m+1}\cdot H_{m}\cdot g_{m}^{-1}$.
In fact, take a solution $v_{m}$ of $N_{m}^{(i,j)}(\alpha)v=0$,
then $u_{m}:=g_{m}^{-1}v_{m}$ is a solution of $M_{m}^{(i,j)}(\alpha)u=0$.
Put $u_{m+1}=H_{m}u_{m}$ and $v_{m+1}:=g_{m+1}u_{m+1}$. Then we
see that $N_{m+1}^{(i,j)}(\alpha)v_{m+1}=0$. If the diagram (\ref{eq:gel-toda-5})
is commutative, $v_{m+1}$ should be obtained as $v_{m+1}=H_{m}'v_{m}$.
Note that 
\[
v_{m+1}=g_{m+1}u_{m+1}=g_{m+1}H_{m}u_{m}=(g_{m+1}\cdot H_{m}\cdot g_{m}^{-1})v_{m}.
\]
So we can determine as $H_{m}'=g_{m+1}\cdot H_{m}\cdot g_{m}^{-1}$.
Similarly one can compute the operator $B_{m}'$ by $g_{m-1}\cdot B_{m}\cdot g_{m}^{-1}$.
Their explicit forms are as follows.

(i) The case $n_{i}=n_{j}=1$. Using (\ref{eq:gel-toda-4}), we have 

\begin{align*}
H_{m}' & =g_{m+1}\cdot H_{m}\cdot g_{m}^{-1}.\\
 & =(x_{0}^{(i)}-x_{0}^{(j)})^{-(\alpha_{0}^{(i)}+m)}\cdot D_{0}^{(j)}\cdot(x_{0}^{(i)}-x_{0}^{(j)})^{\alpha_{0}^{(i)}+m}+\frac{\alpha_{0}^{(j)}-m}{x_{0}^{(i)}-x_{0}^{(j)}}\\
 & =D_{0}^{(j)}+\frac{\alpha_{0}^{(j)}-\alpha_{0}^{(i)}-2m}{x_{0}^{(i)}-x_{0}^{(j)}}.
\end{align*}

\begin{align*}
B_{m}' & =g_{m-1}\cdot B_{m}\cdot g_{m}^{-1}.\\
 & =-\frac{(x_{0}^{(i)}-x_{0}^{(j)})^{2}}{c_{m}^{(i,j)}(\al)}\left\{ (x_{0}^{(i)}-x_{0}^{(j)})^{-(\alpha_{0}^{(i)}+m)}\cdot D_{0}^{(i)}\cdot(x_{0}^{(i)}-x_{0}^{(j)})^{\alpha_{0}^{(i)}+m}-\frac{\alpha_{0}^{(i)}+m}{x_{0}^{(i)}-x_{0}^{(j)}}\right\} \\
 & =-\frac{(x_{0}^{(i)}-x_{0}^{(j)})^{2}}{c_{m}^{(i,j)}(\al)}D_{0}^{(i)}.
\end{align*}

(ii) The case $n_{i}=1,n_{j}\geq2$: 
\begin{align*}
H_{m}' & =g_{m+1}\cdot H_{m}\cdot g_{m}^{-1}.\\
 & =(x_{0}^{(i)}-x_{0}^{(j)})^{-(\alpha_{0}^{(i)}+m)}\cdot D_{n_{j}-1}^{(j)}\cdot(x_{0}^{(i)}-x_{0}^{(j)})^{(\alpha_{0}^{(i)}+m)}+\frac{\alpha_{n_{j}-1}^{(j)}}{x_{0}^{(i)}-x_{0}^{(j)}}\\
 & =D_{n_{j}-1}^{(j)}+\frac{\alpha_{n_{j}-1}^{(j)}}{x_{0}^{(i)}-x_{0}^{(j)}}.
\end{align*}

\begin{align*}
B_{m}' & =g_{m-1}\cdot B_{m}\cdot g_{m}^{-1}.\\
 & =-\frac{(x_{0}^{(i)}-x_{0}^{(j)})^{2}}{c_{m}^{(i,j)}(\al)}\left\{ (x_{0}^{(i)}-x_{0}^{(j)})^{-(\alpha_{0}^{(i)}+m)}\cdot D_{0}^{(i)}\cdot(x_{0}^{(i)}-x_{0}^{(j)})^{\alpha_{0}^{(i)}+m}-\frac{\alpha_{0}^{(i)}+m}{x_{0}^{(i)}-x_{0}^{(j)}}\right\} \\
 & =-\frac{(x_{0}^{(i)}-x_{0}^{(j)})^{2}}{c_{m}^{(i,j)}(\al)}D_{0}^{(i)}.
\end{align*}

(iii) The case $n_{i},n_{j}\geq2$: 
\begin{align*}
H_{m}' & =g_{m+1}\cdot H_{m}\cdot g_{m}^{-1}.\\
 & =\exp\left(-\frac{\alpha_{n_{i}-1}^{(i)}x_{n_{i}-1}^{(i)}}{x_{0}^{(i)}-x_{0}^{(j)}}\right)\left\{ (x_{0}^{(i)}-x_{0}^{(j)})D_{n_{j}-1}^{(j)}+\alpha_{n_{j}-1}^{(j)}\right\} \exp\left(\frac{\alpha_{n_{i}-1}^{(i)}x_{n_{i}-1}^{(i)}}{x_{0}^{(i)}-x_{0}^{(j)}}\right)\\
 & =(x_{0}^{(i)}-x_{0}^{(j)})D_{n_{j}-1}^{(j)}+\alpha_{n_{j}-1}^{(j)}.
\end{align*}

\begin{align*}
B_{m}' & =g_{m-1}\cdot B_{m}\cdot g_{m}^{-1}.\\
 & =\exp\left(-\frac{\alpha_{n_{i}-1}^{(i)}x_{n_{i}-1}^{(i)}}{x_{0}^{(i)}-x_{0}^{(j)}}\right)\frac{1}{c_{m}^{(i,j)}(\al)}\left\{ (x_{0}^{(j)}-x_{0}^{(i)})D_{n_{i}-1}^{(i)}+\alpha_{n_{i}-1}^{(i)}\right\} \exp\left(\frac{\alpha_{n_{i}-1}^{(i)}x_{n_{i}-1}^{(i)}}{x_{0}^{(i)}-x_{0}^{(j)}}\right)\\
 & =\frac{x_{0}^{(j)}-x_{0}^{(i)}}{c_{m}^{(i,j)}(\al)}D_{n_{i}-1}^{(i)}.
\end{align*}
The operators $H_{m}'$ and $B_{m}'$ obtained above are just the
contiguity operators (\ref{eq:darboux-1}) with $C=1$ in the cases
(i) and (ii), and with $C=x_{0}^{(i)}-x_{0}^{(j)}$ in the case (iii).
Note that $C$ is independent of the variables $x_{n_{i}-1}^{(i)},x_{n_{j}-1}^{(j)}$
in all cases. 

For a given $u_{0}(x)\in\mathcal{S}_{0}^{res}(\alpha)$, define $\{u_{m}(x)\}_{m\in\mathbb{Z}}$,
$u_{m}\in\mathcal{S}_{n}^{res}(\alpha)$, by $u_{m+1}=H_{m}u_{m}\;(m\geq0)$
and $u_{m-1}=B_{m}u_{m}\;(m\leq0)$. Putting $u_{m}'(x):=g_{m}(x)u_{m}(x)$
with $g_{m}(x)$ given in (\ref{eq:gel-toda-4}), we have $N_{m}^{(i,j)}(\alpha)u_{m}'=0$
for the Laplace sequence $\{N_{m}^{(i,j)}(\alpha)\}_{m\in\mathbb{Z}}$
such that $u_{m+1}'=H_{m}'u_{m}'$ and $u_{m-1}'=B_{m}'u_{m}'$ for
all $m\in\mathbb{Z}$. To obtain a solution of the 2dTHE 
\begin{equation}
D_{n_{i}-1}^{(i)}D_{n_{j}-1}^{(j)}\log\tau_{m}=\frac{\tau_{m+1}\tau_{m-1}}{\tau_{m}^{2}},\quad m\in\mathbb{Z},\label{eq:gel-toda-9}
\end{equation}
we apply Proposition \ref{prop:Backlund-2} with the seed solution
obtained in Propositions \ref{prop:Lap-2}, \ref{prop:Lap-4} and
\ref{prop:Lap-6}. 

Since the case $n_{i}=n_{j}=1$ is the same as the non-confluent HGF
case and is explained in \cite{hiro-kimu}, we discuss the case $n_{i}=1,n_{j}\geq2$.
In this case, the seed solution for the 2dTHE is 
\[
t_{m}(\bx;\al)=\exp\left(\frac{(\al_{0}^{(i)}+m)\al_{n_{j}-1}^{(j)}x_{n_{j}-1}^{(j)}}{x_{0}^{(i)}-x_{0}^{(j)}}\right)(x_{0}^{(i)}-x_{0}^{(j)})^{-m(m+1)}T_{m}(\al_{0}^{(i)},\al_{n_{j}-1}^{(j)}),
\]
where $T_{m}(\al,\be)$ is that given in Lemma \ref{lem:Lap-3}. Then
taking $u_{m}(x)\in\cS_{m}^{res}(\al)$ as defined above we obtain
the solution $\tau_{m}(x)=t_{m}(\bx;\al)(x_{0}^{(i)}-x_{0}^{(j)})^{-(\alpha_{0}^{(i)}+m)}u_{m}(x)$
to the 2dTHE (\ref{eq:gel-toda-9}). In particular, we can take the
restriction of the Gelfand HGF as $u_{0}(x)$ in the above setting:
$u_{0}(x)=F(\bx;\alpha)$. Taking into account the contiguity relation
for $F(\bx;\al)$, we see that 
\[
u_{m}(x)=(\al_{n_{j}-1}^{(j)})^{m}F(\bx;\alpha+m(\ep^{(i)}-\ep^{(j)})).
\]
The case $n_{i},n_{j}\geq2$ is similarly treated. Summarizing the
above argument, we have the following result. 
\begin{thm}
\label{thm:main} For a partition $\lm=(n_{1},\dots,n_{\ell})$ of
$N$ and a pair $(i,j)$ such that $n_{i}\leq n_{j}$, we can obtain
a solution of the 2dTHE as follows. For a given $u_{0}(x)\in\mathcal{S}_{0}^{res}(\alpha)$,
define the sequence $\{u_{m}(x)\}_{m\in\mathbb{Z}}$ such that $u_{m}(x)\in\cS_{m}^{res}(\al)$
by 
\begin{align*}
u_{m+1} & =H_{m}u_{m}\quad(m\geq0),\quad u_{m-1}=B_{m}u_{m}\quad(n\leq0),
\end{align*}
where $H_{m}:\cS_{m}^{res}(\al)\to\cS_{m+1}^{res}(\al)$ and $B_{m}:\cS_{m}^{res}(\al)\to\cS_{m-1}^{res}(\al)$
are those defined by (\ref{eq:gel-toda-1}) and (\ref{eq:gel-toda-2}),
respectively. Then

(1) $\tau_{m}(x)=t_{m}(\bx;\al)g_{m}(x)u_{m}(x)$ gives a solution
of the 2dTHE 
\[
D_{n_{i}-1}^{(i)}D_{n_{j}-1}^{(j)}\log\tau_{m}=\frac{\tau_{m+1}\tau_{m-1}}{\tau_{m}^{2}},\quad m\in\mathbb{Z},
\]
where $g_{m}(x)$ is that given by (\ref{eq:gel-toda-4}) and $t_{m}(\bx;\al)$
is the seed solution to the 2dTHE given as follows.

(i) If $n_{i}=n_{j}=1$,
\[
t_{m}(\bx;\al)=(x_{0}^{(i)}-x_{0}^{(j)})^{p(\al_{0}^{(i)},\al_{0}^{(j)};m)}T_{m}(\al_{0}^{(i)},\al_{0}^{(j)}),\quad p(\al,\be;m)=(\al+m)(\be-m+1)
\]
with $T_{0}=1,T_{1}=A$, A being an arbitrary constant, and 
\[
T_{m}(\alpha,\beta)=\begin{cases}
A^{m}\prod_{k=0}^{m-1}\left(\prod_{l=1}^{k}p(\alpha,\beta;l)\right), & m\geq2,\\
A^{m}\prod_{k=1}^{|m|}\left(\prod_{l=-k+1}^{0}p(\alpha,\beta;l)\right), & m\leq-1.
\end{cases}
\]

(ii) If $n_{i}=1,n_{j}\geq2$, 

\[
t_{m}(\bx;\al)=\exp\left(\frac{(\al_{0}^{(i)}+m)\al_{n_{j}-1}^{(j)}x_{n_{j}-1}^{(j)}}{x_{0}^{(i)}-x_{0}^{(j)}}\right)(x_{0}^{(i)}-x_{0}^{(j)})^{-m(m+1)}T_{m}(\al_{0}^{(i)},\al_{n_{j}-1}^{(j)})
\]
with $T_{m}(\al,\be)=A^{m}\prod_{k\in\la0,m\ra}[\al]_{k}\cdot\be^{m(m-1)/2}$,
$A$ an arbitrary constant, $\la0,m\ra$ being the set on integers
between $0$ and $m$ including them, and 
\[
[\al]_{k}=\begin{cases}
\G(\al+k)/\G(\al+1), & k\geq1,\\
1 & k=0,\\
\G(\al+1)/\G(\al+1+k) & k\leq-1.
\end{cases}
\]

(iii) If $n_{i},n_{j}\geq2$, 
\[
t_{m}(\bx;\al)=\exp\left(\frac{\al_{n_{i}-1}^{(i)}\al_{n_{j}-1}^{(j)}x_{n_{i}-1}^{(i)}x_{n_{j}-1}^{(j)}}{(x_{0}^{(i)}-x_{0}^{(j)})^{2}}\right)(x_{0}^{(i)}-x_{0}^{(j)})^{-m(m+1)}T_{m}(\al_{n_{i}-1}^{(i)},\al_{n_{j}-1}^{(j)})
\]
with $T_{m}(\al,\be)=A^{m}(\al\be)^{m(m-1)/2}$ and $A$ is an arbitrary
constant.

(2) Let $F(\bx;\alpha)$ be the restriction of the Gelfand HGF $F(z;\al)$
to the slice $X$ defined by $F(\bx;\alpha)=\int_{C}\chi(\vec{s}\bx;\al)ds$.
Then 
\[
\tau_{m}(x)=C_{m}(\al)t_{m}(\bx;\al)g_{m}(x)F(\bx;\alpha+m(\ep^{(i)}-\ep^{(j)}))
\]
gives a solution of the 2dTHE (\ref{eq:gel-toda-9}), where 
\[
C_{m}(\al)=\begin{cases}
\G(\al_{0}^{(j)}+1)/\G(\al_{0}^{(j)}-m+1), & n_{i}=n_{j}=1,\\
(\al_{n_{j}-1}^{(j)})^{m}, & n_{j}\geq2.
\end{cases}
\]
 
\end{thm}

\section{\label{subsec:Classical-HGFs}Appendix}

\subsection{Realization of classical HGFs as Gelfand's HGF}

We give an account that the Gauss HGF and its confluent family can
be understood as the Gelfand HGF on the Grassmannian $\gras(2,4)$,
The members of the confluent family are Kummer's confluent HGF, Bessel
function, Herimite-Weber function and Airy function. They are characterized
by the differential equation on the complex plane $\C$:

\begin{align*}
\text{Gauss}:\quad & x(1-x)y''+\{c-(a+b+1)x\}y'-aby=0,\\
\text{Kummer}:\quad & xy''+(c-x)y'-ay=0,\\
\text{Bessel}:\quad & xy''+(c+1)y'+y=0,\\
\text{Hermite-Weber}:\quad & y''-xy'+cy=0,\\
\text{Airy}:\quad & y''-xy=0.
\end{align*}
For the generic parameters in the equation, the solutions of these
equations are given by the integral with an appropriately chosen path
$\ga$ in the $s$-plane.
\begin{align*}
y(x) & =\int_{\ga}s^{a-1}(1-s)^{c-a-1}(1-xs)^{-b}ds & \text{(Gauss)}\\
y(x) & =\int_{\ga}s^{a-1}(1-s)^{c-a-1}e^{xs}ds, & \text{(Kummer)}\\
y(x) & =\int_{\ga}s^{c-1}e^{xs-\frac{1}{s}}ds, & \text{(Bessel)}\\
y(x) & =\int_{\ga}s^{-c-1}e^{xs-\frac{1}{2}s^{2}}ds, & \text{(Hermite-Weber)}\\
y(x) & =\int_{\ga}e^{xs-\frac{1}{3}s^{3}}ds. & \text{(Airy)}
\end{align*}
For example, the Gauss HGF and Kummer's confluent HGF
\begin{align*}
\hyp 21(a,b,c;x) & =\sum_{m=0}^{\infty}\frac{(a)_{m}(b)_{m}}{(c)m!}x^{m},\\
\hyp 11(a,b,c;x) & =\sum_{m=0}^{\infty}\frac{(a)_{m}}{(c)m!}x^{m}
\end{align*}
are given by taking $\ga$ as $\overrightarrow{0,1}$ in the integrals.
These classical HGFs are realized as the Gelfand HGF on $\gras(2,4)$
according as the partitions $\lm$ of $4$. In fact, we explain that
$(1,1,1,1),(2,1,1),(2,2),(3,1)$ and $(4)$ correspond to Gauss, Kummer,
Bessel and Hermite-Weber and Airy, respectively. For a $\lm$ of $4$,
the necessary data for this link are the maximal abelian group $H_{\lm}$,
the character $\chi_{\lm}$ of $\tH_{\lm}$, the space $Z_{\lm}$,
and the realization $X_{\lm}$ of the quotient space $\GL 2\backslash Z_{\lm}\slash H_{\lm}$.
Here we take a slightly different usage of indices from that used
in the previous sections.

Group $H_{\lm}$:

\begin{align*}
H_{(1,1,1,1)} & =\left\{ \begin{pmatrix}h_{0}\\
 & h_{1}\\
 &  & h_{2}\\
 &  &  & h_{3}
\end{pmatrix}\right\} , & H_{(2,2)} & =\left\{ \begin{pmatrix}h_{0} & h_{1}\\
 & h_{0}\\
 &  & h_{2} & h_{3}\\
 &  &  & h_{2}
\end{pmatrix}\right\} ,\\
H_{(2,1,1)} & =\left\{ \begin{pmatrix}h_{0} & h_{1}\\
 & h_{0}\\
 &  & h_{2}\\
 &  &  & h_{3}
\end{pmatrix}\right\} , & H_{(3,1)} & =\left\{ \begin{pmatrix}h_{0} & h_{1} & h_{2}\\
 & h_{0} & h_{1}\\
 &  & h_{0}\\
 &  &  & h_{3}
\end{pmatrix}\right\} ,\\
H_{(4)} & =\left\{ \begin{pmatrix}h_{0} & h_{1} & h_{2} & h_{3}\\
 & h_{0} & h_{1} & h_{2}\\
 &  & h_{0} & h_{1}\\
 &  &  & h_{0}
\end{pmatrix}\right\} .
\end{align*}

Character $\chi_{\lm}$:
\begin{align*}
\chi_{(1,1,1,1)}(h,\al) & =h_{0}^{\al_{0}}h_{1}^{\al_{1}}h_{2}^{\al_{2}}h_{3}^{\al_{3}},\\
\chi_{(2,1,1)}(h,\al) & =h_{0}^{\al_{0}}\exp\left(\al_{1}\frac{h_{1}}{h_{0}}\right)h_{2}^{\al_{2}}h_{3}^{\al_{3}},\\
\chi_{(2,2)}(h,\al) & =h_{0}^{\al_{0}}\exp\left(\al_{1}\frac{h_{1}}{h_{0}}\right)h_{2}^{\al_{2}}\exp\left(\al_{3}\frac{h_{3}}{h_{2}}\right),\\
\chi_{(3,1)}(h,\al) & =h_{0}^{\al_{0}}\exp\left(\al_{1}\frac{h_{1}}{h_{0}}+\al_{2}\left(\frac{h_{2}}{h_{0}}-\frac{1}{2}\left(\frac{h_{1}}{h_{0}}\right)^{2}\right)\right)h_{3}^{\al_{3}},\\
\chi_{(4)}(h,\al) & =h_{0}^{\al_{0}}\exp\left(\al_{1}\frac{h_{1}}{h_{0}}+\al_{2}\left(\frac{h_{2}}{h_{0}}-\frac{1}{2}\left(\frac{h_{1}}{h_{0}}\right)^{2}\right)\right.\\
 & \qquad+\left.\al_{3}\left(\frac{h_{3}}{h_{0}}-\frac{h_{1}}{h_{0}}\frac{h_{2}}{h_{0}}+\frac{1}{3}\left(\frac{h_{1}}{h_{0}}\right)^{3}\right)\right).
\end{align*}

Matrix space $Z_{\lm}$:
\begin{align*}
Z_{(1,1,1,1)} & =\{(z_{0},z_{1},z_{2},z_{3})\in\mat{2,4}\ \mid\ \det(z_{i},z_{j})\neq0\ (i\neq j)\},\\
Z_{(2,1,1)} & =\left\{ (z_{0},z_{1},z_{2},z_{3})\in\mat{2,4}\ \mid\ \begin{aligned} & \det(z_{0},z_{j})\neq0\ (1\leq j\leq3)\\
 & \det(z_{2},z_{3})\neq0
\end{aligned}
\right\} ,\\
Z_{(2,2)} & =\left\{ (z_{0},z_{1},z_{2},z_{3})\in\mat{2,4}\ \mid\ \begin{aligned} & \det(z_{0},z_{j})\neq0\ (1\leq j\leq2)\\
 & \det(z_{2},z_{3})\neq0
\end{aligned}
\right\} ,\\
Z_{(3,1)} & =\left\{ (z_{0},z_{1},z_{2},z_{3})\in\mat{2,4}\ \mid\ \begin{aligned} & \det(z_{0},z_{1})\neq0\\
 & \det(z_{0},z_{3})\neq0
\end{aligned}
\right\} ,\\
Z_{(4)} & =\left\{ (z_{0},z_{1},z_{2},z_{3})\in\mat{2,4}\ \mid\ \begin{aligned} & \det(z_{0},z_{1})\neq0\end{aligned}
\right\} .
\end{align*}

Realization $X_{\lm}:$

\begin{align*}
X_{(1,1,1,1)} & =\left\{ \begin{pmatrix}1 & 0 & 1 & 1\\
0 & 1 & -1 & -x
\end{pmatrix}\mid x\neq0,1\right\} , & X_{(2,2)} & =\left\{ \begin{pmatrix}1 & 0 & 0 & -1\\
0 & x & 1 & 0
\end{pmatrix}\mid x\neq0\right\} ,\\
X_{(2,1,1)} & =\left\{ \begin{pmatrix}1 & 0 & 0 & 1\\
0 & x & 1 & -1
\end{pmatrix}\mid x\neq0\right\} , & X_{(3,1)} & =\left\{ \begin{pmatrix}1 & 0 & 0 & 0\\
0 & 1 & x & 1
\end{pmatrix}\right\} ,\\
X_{(4)} & =\left\{ \begin{pmatrix}1 & 0 & 0 & 0\\
0 & 1 & 0 & -x
\end{pmatrix}\right\} .
\end{align*}
Then the classical HGF family can be identified as Gelfand's HGF on
$X_{\lm}$ with the parameters chosen appropriately as follows.

(1) $\lm=(1,1,1,1)\leftrightarrow$ Gauss:
\begin{align}
\al & =(\al_{0},\al_{1},\al_{2},\al_{3}):=(b-c,a-1,c-a-1,-b),\nonumber \\
\bx & =(\bx_{0},\bx_{1},\bx_{2},\bx_{3})=\begin{pmatrix}1 & 0 & 1 & 1\\
0 & 1 & -1 & -x
\end{pmatrix},\nonumber \\
F(\bx,\al;C) & =\int_{\ga}(\vec{s}\bx_{0})^{\al_{0}}(\vec{s}\bx_{1})^{\al_{1}}(\vec{s}\bx_{2})^{\al_{2}}(\vec{s}\bx_{3})^{\al_{3}}ds\nonumber \\
 & =\int_{\ga}1^{\al_{0}}s{}^{\al_{1}}(1-s)^{\al_{2}}(1-xs)^{\al_{3}}ds.\label{eq:classical-1}
\end{align}

(2) $\lm=(2,1,1)\leftrightarrow$ Kummer:
\begin{align}
\al & =(\al_{0},\al_{1},\al_{2},\al_{3}):=(-c,1,a-1,c-a-1),\nonumber \\
\bx & =(\bx_{0},\bx_{1},\bx_{2},\bx_{3})=\begin{pmatrix}1 & 0 & 0 & 1\\
0 & x & 1 & -1
\end{pmatrix},\nonumber \\
F(\bx,\al) & =\int_{\ga}(\vec{s}\bx_{0})^{\al_{0}}\exp\left(\al_{1}\frac{\vec{s}\bx_{1}}{\vec{s}\bx_{0}}\right)(\vec{s}\bx_{2})^{\al_{2}}(\vec{s}\bx_{3})^{\al_{3}}ds\nonumber \\
 & =\int_{\ga}1^{\al_{0}}\exp(xs)s^{\al_{2}}(1-s)^{\al_{3}}ds.\label{eq:classical-2}
\end{align}

(3) $\lm=(2,2)\leftrightarrow$ Bessel:
\begin{align}
\al & =(\al_{0},\al_{1},\al_{2},\al_{3}):=(c-1,1,-c-1,1),\nonumber \\
\bx & =(\bx_{0},\bx_{1},\bx_{2},\bx_{3})=\begin{pmatrix}1 & 0 & 0 & -1\\
0 & x & 1 & 0
\end{pmatrix},\nonumber \\
F(\bx,\al) & =\int_{\ga}(\vec{s}\bx_{0})^{\al_{0}}\exp\left(\al_{1}\frac{\vec{s}\bx_{1}}{\vec{s}\bx_{0}}\right)(\vec{s}\bx_{2})^{\al_{2}}\exp\left(\al_{3}\frac{\vec{s}\bx_{3}}{\vec{s}\bx_{2}}\right)ds\nonumber \\
 & =\int_{\ga}1^{\al_{0}}\exp(xs)s^{\al_{2}}\exp(-1/s)ds.\label{eq:classical-be}
\end{align}

(4) $\lm=(3,1)\leftrightarrow$ Hermite-Weber:
\begin{align}
\al & =(\al_{0},\al_{1},\al_{2},\al_{3}):=(a-1,0,1,-a-1),\nonumber \\
\bx & =(\bx_{0},\bx_{1},\bx_{2},\bx_{3})=\begin{pmatrix}1 & 0 & 0 & 0\\
0 & 1 & x & 1
\end{pmatrix},\nonumber \\
F(\bx,\al) & =\int_{\ga}(\vec{s}\bx_{0})^{\al_{0}}\exp\left\{ \frac{\vec{s}\bx_{2}}{\vec{s}\bx_{0}}-\frac{1}{2}\left(\frac{\vec{s}\bx_{1}}{\vec{s}\bx_{0}}\right)^{2}\right\} (\vec{s}\bx_{3})^{\al_{3}}ds\label{eq:classical-h}\\
 & =\int_{\ga}1^{\al_{0}}\exp\left(xs-\frac{1}{2}s^{2}\right)s^{\al_{3}}ds.\nonumber 
\end{align}

(5) $\lm=(4)\leftrightarrow$ Airy:
\begin{align*}
\al & =(\al_{0},\al_{1},\al_{2},\al_{3}):=(-2,0,0,-1),\\
\bx & =(\bx_{0},\bx_{1},\bx_{2},\bx_{3})=\begin{pmatrix}1 & 0 & 0 & 0\\
0 & 1 & 0 & -x
\end{pmatrix},\\
F(\bx,\al) & =\int_{\ga}(\vec{s}\bx_{0})^{\al_{0}}\exp\left\{ \frac{\vec{s}\bx_{3}}{\vec{s}\bx_{0}}-\frac{\vec{s}\bx_{1}}{\vec{s}\bx_{0}}\frac{\vec{s}\bx_{2}}{\vec{s}\bx_{0}}+\frac{1}{3}\left(\frac{\vec{s}\bx_{1}}{\vec{s}\bx_{0}}\right)^{3}\right\} ds\\
 & =\int_{\ga}1^{\al_{0}}\exp\left(xs-\frac{1}{3}s^{3}\right)ds.
\end{align*}

\subsection{\label{subsec:System-reduction}System obtained from the reduction
by $H_{\protect\lm}$}

We give here the list of the 2nd order differential equations obtained
from $\square_{p,q}$, $0\leq p\neq q\leq3$, by the reduction using
the action of $H_{\lm}\;(\lm\neq(4))$ as in Section \ref{sec:Restr-slice}.

\subsubsection{Case $\protect\lm=(2,1,1)$}

Let an element $\bx$ of the slice $X$ is related to $z\in Z_{(2,1,1)}$
as 
\[
\bx=\left(\begin{array}{cccc}
x_{0} & x_{1} & x_{2} & x_{3}\\
1 & 0 & 1 & 1
\end{array}\right)=\left(\begin{array}{ccc}
z_{0,0} & \dots & z_{0,3}\\
z_{1,0} & \dots & z_{1,3}
\end{array}\right)h,\quad h\in H_{(2,1,1)}.
\]

\begin{prop}
Consider the change of unknown $F\mapsto\Phi$ defined by 
\begin{align*}
F(z;\al) & =\chi(\vec{z}_{1};\al)\Phi(x;\al),\\
\chi(\vec{z}_{1};\al) & =(z_{1,0})^{\al_{0}}\exp\left(\al_{1}\frac{z_{1,1}}{z_{1,0}}\right)(z_{1,2})^{\al_{2}}(z_{1,3})^{\al_{3}}.
\end{align*}
Then $\square_{p,q}F=0,\;0\leq p,q\leq3$ give the equations $M_{p,q}(\al)\Phi(x)=0$,
where 
\begin{align*}
M_{0,1} & =x_{1}\pa_{1}^{2}+\al_{1}\pa_{0}-\al_{0}\pa_{1},\\
M_{0,q} & =(x_{0}-x_{q})\pa_{0}\pa_{q}+x_{1}\pa_{1}\pa_{q}+\al_{q}\pa_{0}-\al_{0}\pa_{q},\quad2\leq q\leq3,\\
M_{1,q} & =(x_{0}-x_{q})\pa_{1}\pa_{q}+\al_{q}\pa_{1}-\al_{1}\pa_{q},\quad2\leq q\leq3,\\
M_{2,3} & =(x_{2}-x_{3})\pa_{2}\pa_{3}+\al_{3}\pa_{2}-\al_{2}\pa_{3},.
\end{align*}
\end{prop}

The hyperbolic operators related to the contiguity operators are $M_{1,2},M_{1,3}$
and $M_{2,3}$.
\begin{prop}
Let $\ideal=\ideal(\al)$ be the ideal generated by $\{M_{p,q}(\al)\}$.
Then we can take one of the following sets as a generator of $\ideal$:

\[
\{M_{0,1},M_{1,2},M_{1,3}\},\quad\{M_{0,1},M_{1,2},M_{2,3}\},\quad\{M_{0,1},M_{1,3},M_{2,3}\}.
\]
\end{prop}

\subsubsection{Case $\protect\lm=(2,2)$}

Let an element $\bx$ of the slice $X$ is related to $z\in Z_{(2,2)}$
as 
\[
\bx=\left(\begin{array}{cccc}
x_{0} & x_{1} & x_{2} & x_{3}\\
1 & 0 & 1 & 0
\end{array}\right)=\left(\begin{array}{ccc}
z_{0,0} & \dots & z_{0,3}\\
z_{1,0} & \dots & z_{1,3}
\end{array}\right)h\quad h\in H_{(2,2)}.
\]

\begin{prop}
Consider the change of unknown $F\mapsto\Phi$ by 
\begin{align*}
F(z;\al) & =\chi(\vec{z}_{1};\al)\Phi(x;\al),\\
\chi(\vec{z}_{1};\al) & =(z_{1,0})^{\al_{0}}\exp\left(\al_{1}\frac{z_{1,1}}{z_{1,0}}\right)(z_{1,2})^{\al_{2}}\exp\left(\al_{3}\frac{z_{1,3}}{z_{1,2}}\right).
\end{align*}
Then $\square_{p,q}F=0,\;0\leq p\neq q\leq N$ give the equations
$M_{p,q}(\al)\Phi(x)=0$, where 
\begin{align*}
M_{0,1} & =x_{1}\pa_{1}^{2}+\al_{1}\pa_{0}-\al_{0}\pa_{1},\\
M_{0,2} & =\left(x_{0}-x_{2}\right)\pa_{0}\pa_{2}+x_{1}\pa_{1}\pa_{2}-x_{3}\pa_{0}\pa_{3}+\al_{2}\pa_{0}-\al_{0}\pa_{2},\\
M_{0,3} & =\left(x_{0}-x_{2}\right)\pa_{0}\pa_{3}+x_{1}\pa_{1}\pa_{3}+\al_{3}\pa_{0}-\al_{0}\pa_{3},\\
M_{1,2} & =\left(x_{0}-x_{2}\right)\pa_{1}\pa_{2}-x_{3}\pa_{1}\pa_{3}+\al_{2}\pa_{1}-\al_{1}\pa_{2},\\
M_{1,3} & =\left(x_{0}-x_{2}\right)\pa_{1}\pa_{3}+\al_{3}\pa_{1}-\al_{1}\pa_{3},\\
M_{2,3} & =x_{3}\pa_{3}^{2}+\al_{3}\pa_{2}-\al_{2}\pa_{3}.
\end{align*}
\end{prop}

The hyperbolic operator related to the contiguity operators is $M_{1,3}$. 
\begin{prop}
Let $\ideal=\ideal(\al)$ be the ideal of generated by the above operators.
We have 

\begin{align*}
\left(x_{0}-x_{2}\right)\pa_{3}M_{0,1}-x_{1}\pa_{1}M_{1,3} & =-\al_{0}M_{1,3}+\al_{1}M_{0,3}-\al_{3}M_{0,1},\\
x_{3}\pa_{3}M_{1,3}-\left(x_{0}-x_{2}\right)\pa_{1}M_{2,3} & =-\al_{1}M_{2,3}+\al_{2}M_{1,3}-\al_{3}M_{1,2},\\
\left\{ \left(x_{0}-x_{2}\right)\pa_{2}-x_{3}\pa_{3}\right\} M_{0,1}-x_{1}\pa_{1}M_{1,2} & =-\al_{0}M_{1,2}+\al_{1}M_{0,2}-\al_{2}M_{0,1}.
\end{align*}
It follows that $M_{0,3}\in\la M_{0,1},M_{1,3}\ra,\quad M_{1,2}\in\la M_{1,3},M_{2,3}\ra,\quad M_{0,2}\in\la M_{0,1},M_{1,2}\ra$,
where $\la M_{0,1},M_{1,3}\ra$ is the ideal generated by $M_{0,1},M_{1,3}$.
Hence we can take $G:=\{M_{0,1},M_{1,3},M_{2,3}\}$ as a generator
of $\ideal$.
\end{prop}

\subsubsection{Case $\protect\lm=(3,1)$}

Let an element $\bx$ of the slice $X$ is related to $z\in Z_{(3,1)}$
as 
\[
\bx=\left(\begin{array}{cccc}
x_{0} & x_{1} & x_{2} & x_{3}\\
1 & 0 & 0 & 1
\end{array}\right)=\left(\begin{array}{ccc}
z_{0,0} & \dots & z_{0,3}\\
z_{1,0} & \dots & z_{1,3}
\end{array}\right)h,\quad h\in H_{(3,1)}.
\]

\begin{prop}
Consider the change of unknown $F\mapsto\Phi$ defined by 
\begin{align*}
F(z;\al) & =\chi(\vec{z}_{1};\al)\Phi(x;\al),\\
\chi(\vec{z}_{1};\al) & =(z_{1,0})^{\al_{0}}\exp\left(\al_{1}\frac{z_{1,1}}{z_{1,0}}+\al_{2}\left(\frac{z_{1,2}}{z_{1,0}}-\frac{1}{2}\left(\frac{z_{1,1}}{z_{1,0}}\right)^{2}\right)\right)(z_{1,3})^{\al_{3}}.
\end{align*}
Then $\square_{p,q}F=0,\;0\leq p\neq q\leq3$ give the equations $M_{p,q}(\al)\Phi(x)=0$,
where 
\begin{align*}
M_{0,1} & =x_{1}\pa_{1}^{2}-x_{1}\pa_{0}\pa_{2}+x_{2}\pa_{1}\pa_{2}+\al_{1}\pa_{0}-\al_{0}\pa_{1}\\
M_{0,2} & =x_{1}\pa_{1}\pa_{2}+x_{2}\pa_{2}^{2}+\al_{2}\pa_{0}-\al_{0}\pa_{2},\\
M_{1,2} & =x_{1}\pa_{2}^{2}+\al_{2}\pa_{1}-\al_{1}\pa_{2},\\
M_{0,3} & =(x_{0}-x_{3})\pa_{0}\pa_{3}+x_{1}\pa_{1}\pa_{3}+x_{2}\pa_{2}\pa_{3}+\al_{3}\pa_{0}-\al_{0}\pa_{3},\\
M_{1,3} & =(x_{0}-x_{3})\pa_{1}\pa_{3}+x_{1}\pa_{2}\pa_{3}+\al_{3}\pa_{1}-\al_{1}\pa,\\
M_{2,3} & =(x_{0}-x_{3})\pa_{2}\pa_{3}+\al_{3}\pa_{2}-\al_{2}\pa_{3}.
\end{align*}
\end{prop}

\begin{prop}
Let $\ideal=\ideal(\al)$ be the ideal of generated by the above operators.
Then we have

\begin{align*}
x_{1}\pa_{2}M_{0,2}-(x_{1}\pa_{1}+x_{2}\pa_{2})M_{1,2} & =-\al_{0}M_{1,2}+\al_{1}M_{0,2}-\al_{2}M_{0,1},\\
(x_{0}-x_{3})\pa_{3}M_{1,2}-x_{1}\pa_{2}M_{2,3} & =-\al_{1}M_{2,3}+\al_{2}M_{1,3}-\al_{3}M_{1,2},\\
(x_{0}-x_{3})\pa_{3}M_{0,2}-(x_{1}\pa_{1}+x_{2}\pa_{2})M_{2,3} & =-\al_{0}M_{2,3}+\al_{2}M_{0,3}-\al_{3}M_{0,2}.
\end{align*}
It follows that $M_{0,1}\in\la M_{0,2},M_{1,2}\ra,\quad M_{1,3}\in\la M_{1,2},M_{2,3}\ra,\quad M_{0,3}\in\la M_{0,2},M_{2,3}\ra$.
Hence we can take $G:=\{M_{0,2},M_{1,2},M_{2,3}\}$ as a generator
of $\ideal$ .
\end{prop}


\begin{thebibliography}{10}
\bibitem{Appell-2} P. Appell, J. Kamp\'e de F\'eriet, Fonction
hyperg\'eom\'etrique et hypersph\'eriques, Gauthier Villars, Paris,
(1926).

\bibitem{Darboux} G. Darboux, Leçon sur la Th\'eorie G\'en\'eral
des Surfaces, t.II, t.IV, Chelsea, New York, (1972).

\bibitem{Erdelyi} A. Erdelyi et al., Higher transcendental functions,
vol I,II, R. E. Krieger Pub. Co (1981).

\bibitem{Gelfand} I. M. Gelfand, General theory of hypergeometric
functions, Soviet Math. Dokl. 33 (1986), 9--13.

\bibitem{Hirota} R. Hirota, Y. Ohta and J. Satsuma, Wronskian structures
of solutions for soliton equations, Progr. Theoret. Phys. Suppl. No.
94 (1988), 59--72.

\bibitem{IKSY} K. Iwasaki, H. Kimura, S. Shimomura and M. Yoshida,
From Gauss to Painlev\'e, Vieweg Verlag, (1991).

\bibitem{kame-1} Y. Kametaka, Hypergeometric solutions of Toda equation,
S\^urikaisekikenky\^usho K\^okyuroku 554 (1985), 26-46.

\bibitem{kame-2} Y. Kametaka, On the Euler-Poisson-Darboux equation
and the Toda equation I, II, Proc. Japan Acad. 60A (1984), 145-148,
181-184.

\bibitem{hiro-kimu}H. Kimura, Gelfand hypergeometric functions as
solutions to the 2-dimensional Toda-Hirota equations, preprint arXiv:2506.04638
(2025).

\bibitem{Kimura-Haraoka}H. Kimura, Y. Haraoka and K. Takano, The
generalized confluent hypergeometric functions, Proc. Japan Acad.
68 (1992), 290--295.

\bibitem{Kimura-H-T} H. Kimura, Y. Haraoka and K. Takano, On contiguity
relations of the confluent hypergeometric systems, Proc. Japan Acad.
Ser. A Math. Sci. 70 (1994), no. 2, 47--49.

\bibitem{Nakamura-A} A. Nakamura, Toda equation and its solutions
in special functions. J. Phys. Soc. Japan 65 (1996), no. 6, 1589--1597.

\bibitem{Okamoto-1} K. Okamoto, Sur les \'echelles associ\'ees
aux fonctions sp\'eciales et l'\'equation de Toda, J. Fac. Sci.
Univ. Tokyo Sect. IA Math. 34 (1987), no. 3, 709--740.

\bibitem{Okamoto-2} K. Okamoto, \'Echelles et l'\'equation de Toda,
Publ. Inst. Rech. Math. Av., Universit\'e Louis Pasteur, 1988, 19--38.

\bibitem{Olshanetsky} M.A. Olshanetsky and A.M. Perelomov, Explicit
solutions of classical generalized Toda models, Invent. Math. 54 (1979),
261-269.

\bibitem{Pham} F. Pham, Introduction \`a l'\'etude topologique
de singularit\'es de Landau, M\'emoire Sci. Math. 164, Gauthiers
Villars (1967).

\bibitem{Satuma} T. Tokihiro, J. Satsuma and R. Willox, On special
function solutions to nonlinear integrable equations. Physics Letters
A 236 (1997) 21-29.

\end{thebibliography}
\end{document}